\documentclass[a4paper,11pt,reqno]{amsart}

\usepackage[utf8]{inputenc}
\usepackage[T1]{fontenc}
\usepackage{lmodern}
\usepackage[english]{babel}
\usepackage{microtype}

\usepackage{amsmath,amssymb,amsfonts,amsthm,esint}
\usepackage{mathtools,accents}
\usepackage{mathrsfs}
\usepackage{aliascnt}
\usepackage{braket}
\usepackage{bm}

\usepackage[a4paper,margin=3cm]{geometry}
\usepackage[citecolor=blue,colorlinks]{hyperref}

\usepackage{enumerate}
\usepackage{xcolor}
\makeatletter
\g@addto@macro\@floatboxreset\centering
\makeatother

\makeatletter
\def\newaliasedtheorem#1[#2]#3{
	\newaliascnt{#1@alt}{#2}
	\newtheorem{#1}[#1@alt]{#3}
	\expandafter\newcommand\csname #1@altname\endcsname{#3}
}
\makeatother

\numberwithin{equation}{section}

\newtheoremstyle{slanted}{\topsep}{\topsep}{\slshape}{}{\bfseries}{.}{.5em}{}

\theoremstyle{plain}
\newtheorem{theorem}{Theorem}[section]
\newaliasedtheorem{proposition}[theorem]{Proposition}
\newaliasedtheorem{lemma}[theorem]{Lemma}
\newaliasedtheorem{corollary}[theorem]{Corollary}
\newaliasedtheorem{counterexample}[theorem]{Counterexample}

\theoremstyle{definition}
\newaliasedtheorem{definition}[theorem]{Definition}
\newaliasedtheorem{question}[theorem]{Question}
\newaliasedtheorem{openquestion}[theorem]{Open Question}
\newaliasedtheorem{conjecture}[theorem]{Conjecture}

\theoremstyle{remark}
\newaliasedtheorem{remark}[theorem]{Remark}
\newaliasedtheorem{example}[theorem]{Example}

\newcommand{\setN}{\mathbb{N}}

\newcommand{\setR}{\mathbb{R}}

\newcommand{\eps}{\varepsilon}

\let\altphi\phi
\let\phi\varphi
\let\varphi\altphi
\let\altphi\undefined

\newcommand{\abs}[1]{\left\lvert#1\right\rvert}
\newcommand{\norm}[1]{\left\lVert#1\right\rVert}

\newcommand{\weakto}{\rightharpoonup}

\newcommand{\Id}{\mathrm{Id}}
\DeclareMathOperator{\tr}{tr}

\DeclareMathOperator{\Hess}{Hess}

\newcommand{\di}{\mathop{}\!\mathrm{d}}

\newcommand{\res}{\mathop{\hbox{\vrule height 7pt width .5pt depth 0pt
			\vrule height .5pt width 6pt depth 0pt}}\nolimits}

\DeclareMathOperator{\Per}{Per}
\DeclareMathOperator{\Ric}{Ric}

\newcommand{\haus}{\mathscr{H}}
\newcommand{\Leb}{\mathscr{L}}

\newcommand{\dist}{\mathsf{d}}

\newcommand{\meas}{\mathfrak{m}}

\DeclareMathOperator{\CD}{CD}
\DeclareMathOperator{\RCD}{RCD}

\newfont{\tmpf}{cmsy10 scaled 2500}

\def\XXint#1#2#3{{\setbox0=\hbox{$#1{#2#3}{\int}$ }
		\vcenter{\hbox{$#2#3$ }}\kern-.6\wd0}}

\begin{document}
	
\title{The metric measure boundary of spaces with Ricci curvature bounded below}

\author{Elia Bru\`{e}, Andrea Mondino, and Daniele Semola}

\address{School of Mathematics, Institute for Advanced Study\\
	1 Einstein Dr.\\
	Princeton NJ 05840\\
	U.S.A.}
\email{elia.brue@math.ias.edu}

\address{Mathematical Institute, University of Oxford\\
	Oxford OX2 6GG \\
	United Kingdom}
\email{Andrea.Mondino@maths.ox.ac.uk}

\address{Mathematical Institute, University of Oxford\\
	Oxford OX2 6GG \\
	United Kingdom}
\email{Daniele.Semola@maths.ox.ac.uk}

\maketitle

\begin{abstract}

We solve a conjecture raised by Kapovitch, Lytchak and Petrunin in \cite{LytchakKapovitchPetrunin21} by showing that the \textit{metric measure boundary} is vanishing on any $\RCD(K,N)$ space $(X,\dist,\haus^N)$ without boundary.
Our result, combined with \cite{LytchakKapovitchPetrunin21}, settles an open question about the existence of infinite geodesics on Alexandrov spaces without boundary raised by Perelman and Petrunin in 1996.

\end{abstract}

\tableofcontents

\section{Introduction and main results}

We study the {\it metric measure boundary} of noncollapsed spaces with Ricci curvature bounded from below. We work within the framework of $\RCD(K,N)$ spaces, a class of infinitesimally Hilbertian metric measure spaces verifying the synthetic Curvature-Dimension condition $\CD(K,N)$ from \cite{Sturm06a,Sturm06b,LottVillani}. We assume the reader to be familiar with the $\RCD$ theory addressing to \cite{AmbrosioGigliSavare14,Gigli15,AmbrosioGigliMondinoRajala15,AmbrosioGigliSavare15,ErbarKuwadaSturm15,AmbrosioMondinoSavare19, CavallettiMilman21} for the basic background. Further references for the statements relevant to our purposes will be pointed out subsequently in the note.

\smallskip

Given an $\RCD(-(N-1),N)$ space $(X,\dist, \haus^N)$ and $r>0$, we introduce
\begin{equation}\label{eq:defmur}
\mu_r(\di x)
:= \frac{1}{r} \mathcal{V}_r (\di x)
= \frac{1}{r}\left(1-\frac{\haus^N(B_r(x))}{\omega_Nr^N}\right)\haus^N (\di x)\, ,
\end{equation}
where $\mathcal{V}_r$ is the \emph{deviation measure} in the terminology of \cite{LytchakKapovitchPetrunin21}, $\haus^N$ is the $N$-dimensional Hausdorff measure, $\omega_N$ is the volume of the unit ball in $\setR^N$ and $B_r(x)$ denotes the open ball of radius $r$ centered at $x\in X$.
\smallskip

If $(X,\dist)$ is isometric to a smooth $N$-dimensional Riemannian manifold $(M,g)$ without boundary, it is a classical result that
\begin{equation}\label{eq:}
 \frac{\haus^N(B_r(x))}{\omega_Nr^N}=1-\frac{\mathrm{Scal}(x)}{6(N+2)}r^2+O(r^4)\, ,\quad\text{as $r\downarrow 0$}\, ,
\end{equation}
at any point $x$, where $\mathrm{Scal}(x)$ denotes the scalar curvature of $(M,g)$ at $x$. Then it is a standard computation to show that 
$$\mu_r\to \gamma (N)\haus^{N-1}\res\partial X \quad \text{ weakly as measures as $r\downarrow 0$}\, ,
$$ 
when $(X,\dist)$ is isometric to a smooth Riemannian manifold with boundary $\partial X$. Here $\haus^{N-1}$ is the $(N-1)$-dimensional volume measure and $\gamma(N)>0$ is a universal constant depending only on the dimension (see \eqref{eq:defgammaNintro} below for the explicit expression).

This observation motivates the following definition, see \cite[Definition 1.5]{LytchakKapovitchPetrunin21}.

\begin{definition}
We say that an $\RCD(-(N-1),N)$ space $(X,\dist,\haus^N)$ has locally finite metric measure boundary if the family of Radon measures $\mu_r$ as in \eqref{eq:defmur} is locally uniformly bounded for $0<r\le 1$. If there exists a weak limit $\mu=\lim_{r\downarrow 0}\mu_r$, then we shall call $\mu$ the metric measure boundary of $(X,\dist,\haus^N)$. Moreover, if $\mu=0$, we shall say that $X$ has vanishing metric measure boundary. 
\end{definition}

We recall that the boundary of an $\RCD(-(N-1),N)$ metric measure space $(X,\dist,\haus^N)$ is defined as the closure of the top dimensional singular stratum
\begin{equation}
\mathcal{S}^{N-1}\setminus\mathcal{S}^{N-2}
:= \{x\in X\, : \, \textit{the half space $\setR^N_+$ is a tangent cone at $x$}\}\, .
\end{equation}
When $\mathcal{S}^{N-1}\setminus\mathcal{S}^{N-2}=\emptyset$, we say that $X$ has no boundary. 
We refer to \cite{BrueNaberSemola22} (see also the previous \cite{DePhilippisGigli18,KapovitchMondino19}) for an account on regularity and stability of boundaries of $\RCD$ spaces.

Our goal is to prove the following.

\begin{theorem}\label{thm:mms0}
Let $N\ge 1$ and $(X,\dist,\haus^N)$ be an $\RCD(-(N-1),N)$ metric measure space. Let $p\in X$ be such that $B_2(p)\cap \partial X = \emptyset $ and $\haus^N(B_1(p)) \ge v>0$, then
\begin{equation}\label{eq:bound effective}
	|\mu_r|(B_1(p)) \le C(N,v) \, ,
	\qquad
	\text{for any $r>0$} \, ,
\end{equation}
and $\lim_{r\downarrow 0} |\mu_r|(B_1(p)) = 0$.
In particular, if $X$ has empty boundary then it has vanishing metric measure boundary.
\end{theorem}

The effective bound \eqref{eq:bound effective} is new even when $(X,\dist)$ is isometric to a smooth $N$-dimensional manifold satisfying $\Ric \ge -(N-1)$. However, the most relevant outcome of \autoref{thm:mms0} is the second conclusion, showing that $\RCD$ spaces $(X,\dist,\haus^N)$ with empty boundary have vanishing metric measure boundary. This implication was unknown even in the setting of Alexandrov spaces, where it was conjectured to be true by Kapovitch-Lytchak-Petrunin \cite{LytchakKapovitchPetrunin21}. By the compatibility between the theory of Alexandrov spaces with sectional curvature bounded from below and the $\RCD$ theory, see \cite{Petrunin11} and the subsequent \cite{ZhangZhu10}, \autoref{thm:mms0} fully solves this conjecture.

\medskip

We are able to control the metric measure boundary also for $\RCD(-(N-1),N)$ spaces $(X,\dist,\haus^N)$ with boundary under an extra assumption. The latter is always satisfied on Alexandrov spaces with sectional curvature bounded below and on noncollapsing limits of manifolds with convex boundary and Ricci curvature uniformly bounded below. 
\\We shall denote
\begin{equation*}
	V_r(s) : = \frac{\Leb^N(B_r((0,s))\cap \{x_N>0\})}{\omega_N r^N} \, ,
\end{equation*}
where $(0,s)\in \setR^{N-1}\times \setR_+$. Moreover, we set 
\begin{equation}\label{eq:defgammaNintro}
\gamma(N) := \omega_{N-1}\int_0^1(1-V_1(t))\di t\, .
\end{equation}

\begin{theorem}\label{thm:boundary intro}
Let $(X,\dist,\haus^N)$ be either an Alexandrov space with (sectional) curvature $\geq -1$ or a noncollapsed limit of manifolds with convex boundary and $\Ric \ge -(N-1)$ in the interior. Let $p\in X$ be such that $\haus^N(B_1(p))\ge v>0$. Then
\begin{equation}
	\mu_r(B_2(p))\le C(N,v) \, ,
	\quad
	\text{for any $r>0$} \, .
\end{equation}
Moreover, 
\begin{equation}
	\mu_r \weakto \gamma(N)\haus^{N-1}\res \partial X\, ,
	\qquad
	\text{as $r\downarrow 0$}\, ,
\end{equation}
where $\gamma(N)>0$ is the constant defined in \eqref{eq:defgammaNintro}.
\end{theorem}

In other words, the metric measure boundary coincides with the boundary measure. We refer to \autoref{sec:spaces with boundary} for a more general statement.
\medskip

On a complete Riemannian manifold without boundary all the geodesics extend for all times, while in the presence of boundary the amount of geodesics that terminate (on the boundary) is measured by its \emph{size}. This is of course too much to hope for on general metric spaces.
However, as shown in \cite[Theorem 1.6]{LytchakKapovitchPetrunin21}, when the metric measure boundary is vanishing on an Alexandrov space with sectional curvature bounded from below, then there are \emph{many} infinite geodesics.  We refer to \cite[Section 3]{LytchakKapovitchPetrunin21} for the definitions of tangent bundle, geodesic flow and Liouville measure in the setting of Alexandrov spaces. An immediate application of \autoref{thm:mms0}, when combined with \cite[Theorem 1.6]{LytchakKapovitchPetrunin21}, is the following.

\begin{theorem}
	Let $(X,\dist)$ be an Alexandrov space with empty boundary. 
	Then almost each direction of the tangent bundle $TX$ is the starting direction of an infinite geodesic. Moreover, the geodesic flow preserves the Liouville measure on $TX$.	
\end{theorem}

In particular, the above gives an affirmative answer to an open question raised by Perelman-Petrunin \cite{PerelmanPetrunin96} about the existence of infinite geodesics on Alexandrov spaces with empty topological boundary.

\subsection*{Outline of proof}

The main challenge in the study of the metric measure boundary is to control the mass of inner balls, i.e. balls located sufficiently far away from the boundary. This is the aim of \autoref{thm:mms0}, whose proof occupies the first five sections of this paper and requires several new ideas. 
Once \autoref{thm:mms0} is established, \autoref{thm:boundary intro} follows from a careful analysis of boundary balls. The latter is outlined in \autoref{sec:spaces with boundary}.

\medskip

	Let us now describe the proof of \autoref{thm:mms0}. Given a ball $B_1(p)\subset X$ such that $B_2(p)\cap \partial X = \emptyset$, we aim at finding uniform bounds on the family of approximating measures 
	\begin{equation}\label{eq:goal 1}
		\mu_r(B_1(p)) = \frac{1}{r}\int_{B_1(p)} \left(1-\frac{\haus^N(B_r(x))}{\omega_N r^N}\right) \di \haus^N(x) \, 
	\end{equation}
    and at showing that
    \begin{equation}\label{eq:goal 2}
    	\lim_{r\downarrow 0} \frac{1}{r}\int_{B_1(p)} \left(1-\frac{\haus^N(B_r(x))}{\omega_N r^N}\right) \di \haus^N(x)  = 0 \, .
    \end{equation}
    Morally, \eqref{eq:goal 2} amounts to say that the identity 
    \begin{equation}\label{eq: goal}
    	\frac{\haus^N(B_r(x))}{\omega_N r^N} = 1 + o(r)
    \end{equation}
    holds in average on $B_1(p)$. The Bishop-Gromov inequality says that the limit
	\begin{equation}\label{eq:volrat}
		\lim_{r\downarrow 0}\frac{\haus^N(B_r(x))}{\omega_Nr^N}
	\end{equation}
	exists for all points. Moreover, its value is $1$ if and only if $x$ is a regular point, i.e. its tangent cone is Euclidean. In particular the limit is different from $1$ only at singular points, which are a set of Hausdorff dimension less than $(N-2)$ if there is no boundary. This is a completely non trivial statement, although now classical, as it requires the volume convergence theorem and the basic regularity theory for noncollapsed spaces with lower Ricci bounds \cite{Colding97,CheegerColding97,DePhilippisGigli18}. Analogous statements were known for Alexandrov spaces with curvature bounded from below since \cite{BuragoGromovPerelman92}.

	\smallskip

	The proof of \eqref{eq:goal 1} and \eqref{eq:goal 2} is based on three main ingredients:
	\begin{itemize}
		\item[(1)] a new quantitative volume convergence result via $\delta$-splitting maps, see \autoref{prop:quantitativevolcE};
		
		\item[(2)] an $\eps$-regularity theorem, see \autoref{prop:boundary on regular balls}, stating (roughly) that for balls which are sufficiently close to the Euclidean ball in the Gromov-Hausdorff topology, the approximating measure $\mu_r$ as in \eqref{eq:defmur} is small;

		\item[(3)] a series of quantitative covering arguments.
   \end{itemize}

The first two ingredients are the main contributions of the present work. We believe that they are of independent interest and have a strong potential for future applications in the study of spaces with Ricci curvature bounded below.

The ingredient (3) comes from the recent \cite{BrueNaberSemola22}, see \autoref{thm:decompositiontheorem} for the precise statement and \cite{JiangNaber16,CheegerJiangNaber21,LytchakKapovitchPetrunin21,LiNaber} for earlier versions in different contexts. It is used to globalize local bounds obtained out of (2) by summing up good scale invariant bounds on almost Euclidean balls.

\subsubsection*{Quantitative volume convergence}

The starting point of our analysis is (1). It provides a quantitative control on the volume of almost Euclidean balls, i.e. balls $B_1(p)\subset X$ such that 
\begin{equation}
	\dist_{GH}\left(B_5(p), B_5^{\setR^N}(0)\right) \le \delta \ll 1 \, ,
\end{equation}
in terms of $\delta$-splitting maps $u:B_5(p)\to\setR^N$. The latter are integrally good approximations of the canonical coordinates of $\setR^N$ satisfying
\begin{equation}
	\Delta u_i=0\, ,\quad \fint_{B_5(p)}\abs{\nabla u_i\cdot\nabla u_j-\delta_{ij}}\di\haus^N\le \delta\, ,\quad\fint_{B_5(p)}\abs{\Hess u_i}^2\di\haus^N\le \delta\, ,
\end{equation}
see \cite{CheegerColding96,CheegerColding97,CheegerNaber15,CheegerJiangNaber21} for the theory on smooth manifolds and Ricci limit spaces and the subsequent \cite{BrueNaberSemola22} for the present setting. The key inequality proven in \autoref{prop:quantitativevolcE} reads as
\begin{equation}\label{eq:volume conv intro}
	\begin{split}
	  	\abs{1-
	  		\frac{\haus^N(B_r(x))}{\omega_N r^N}}
	  	\le C(N)\left( r^2 
	  	+\int_0^r\fint_{B_{4t}(x)}\abs{\nabla u_i\cdot\nabla u_j-\delta_{ij}}\di\haus^N\frac{\di t}{t}
	  	\right)\, ,
	\end{split}
\end{equation}
at any regular point $x\in B_5(p)$, for any $r<5$. The term appearing in the right hand side measures to what extent $u:B_5(p)\to \setR^N$ well-approximates the Euclidean coordinates at any scale $r\in (0,5)$ around $x$. %We remark that following the classical arguments for the proof of the volume convergence \cite{Colding97,CheegerColding2000b,Cheeger01} a power $\alpha(N)<1$ would appear at the right hand side in \eqref{eq:volume conv intro}, while for our applications the linear dependence is fundamental.

\medskip

In order to prove \eqref{eq:volume conv intro}, we use the components of the splitting map to construct an approximate solution of the equations
\begin{equation}\label{eq:functcone1}
	\Delta r^2=2N\, ,\quad \abs{\nabla r}=1\, ,
\end{equation}
with $r\ge 0$ and $r(x)=0$. The approximate solution is obtained as $r^2:=\sum_iu_i^2$, after normalizing so that $u(x)=0$, and the right hand side in \eqref{eq:rate} controls the precision of this approximation, see \autoref{lemma:bound r}. Then the idea is that when $\Ric\ge 0$ the existence of a solution of \eqref{eq:functcone1} would force the volume ratio to be constant along scales. In \autoref{lemma:derv} we prove an effective version of this where errors are taken into account quantitatively. 
\medskip

We remark that, following the proofs of the volume convergence in \cite{Colding97,CheegerColding2000b,Cheeger01}, one would get an estimate
\begin{equation}
\abs{1-
	  		\frac{\haus^N(B_r(x))}{\omega_N r^N}}
	  	\le C(N)\left(r^2+\fint_{B_{4r}(x)}\abs{\nabla u_i\cdot\nabla u_j-\delta_{ij}}\di\haus^N\right)^{\alpha(N)}\, 
\end{equation}
for some $\alpha(N)<1$, while for our applications it is fundamental to have a linear dependence at the right hand side. This is achieved by estimating the derivative at any scale,
     \begin{align}
  \nonumber	- \frac{\di}{\di t} \left( \frac{\haus^{N}(B_t(x))}{t^N}   \right)
     	\le& 
     	C(N)t^{-N} \int_{\partial B_t(x)} \abs{\nabla u_i\cdot\nabla u_j-\delta_{ij}} \di \haus^{N-1} \\
	&+ C(N) t^{-N - 1}\int_{B_{4t}(x)} \abs{\nabla u_i\cdot\nabla u_j-\delta_{ij}} \di \haus^N \, ,
     \end{align}
    see \autoref{cor:gapBGharm},  and then integrating with respect to the scale. The improved dependence comes at the price of considering a multi-scale object at the right hand side.

%
%{\color{red}
%Notice that here we are reversing the implications with respect to the classical \cite{CheegerColding96} and the recent \cite{CheegerJiangNaber21}. Indeed, in \cite{CheegerColding96} almost constancy of the volume ratios is used to build approximate solutions of \eqref{eq:functcone1} and in \cite{CheegerJiangNaber21} good approximate solutions of \eqref{eq:functcone1} are used to build good splitting functions. Here we exploit the splitting functions to construct good approximate solutions of \eqref{eq:functcone1} and to prove eventually that the volume ratio is converging polynomially to $1$.
%}

\subsubsection*{A new $\eps$-regularity theorem}
Let us now outline the ingredient (2).
The $\eps$-regularity theorem, \autoref{prop:boundary on regular balls}, amounts to show that the scale invariant volume ratio in \eqref{eq:volrat} converges to $1$ at the quantitative rate $o(r)$ \emph{in average} on a ball $B_{10}(p)$ which is sufficiently close to the Euclidean ball $B^{\setR^N}_{10}(0)\subset\setR^N$ in the Gromov-Hausdorff sense.\\ 
In order to prove it, we employ the quantitative volume bound \eqref{eq:volume conv intro}.
There are two key points to take into account dealing with harmonic splitting maps in the present setting:
\begin{itemize}
\item[(a)] they cannot be bi-Lipschitz in general, as they do not remain $\delta$-splitting maps when restricted to smaller balls $B_r(x)\subset B_5(p)$.
\item[(b)] they have good $L^2$ integral controls on their Hessians.
\end{itemize}
This is in contrast with distance coordinates in Alexandrov geometry, that are biLipschitz but have good controls only on the total variation of their measure valued Hessians, see \cite{Perelman95}. On the one hand, (a) makes controlling the metric measure boundary much more delicate than in the Alexandrov case. On the other hand, (b) is where the crucial gain with respect to the previous \cite{LytchakKapovitchPetrunin21} appears. Indeed, the $L^p$ integrability for $p\ge 1$ allows to show that the metric measure boundary cannot concentrate on a set negligible with respect to $\haus^N$. At this point, it will be sufficient to prove that the rate of convergence to $1$ in \eqref{eq:volrat} is $o(r)$ at $\haus^N$-a.e. point.

The key observation to deal with (a) is that, even though a splitting map can degenerate, it remains quantitatively well behaved away from a set $E\subset B_5(p)$ for which there exists a covering
\begin{equation}\label{eq:content}
E\subset \bigcup_{i}B_{r_i}(x_i)\, ,\quad \text{with}\quad \sum_ir_i^{N-1}\le \delta'\, ,
\end{equation}
and $\delta'\to 0$ as $\delta\to 0$. Moreover, on $B_5(p)\setminus E$ the splitting map becomes polynomially better and better when restricted to smaller balls, after composition with a linear transformation close to the identity in the image. Namely there exists a linear application $A_x:\setR^N\to\setR^N$ with $\abs{A_x-\mathrm{Id}}\le C(N)\delta'$ for which, setting $v:=A_x\circ u:B_5(p)\to\setR^N$, it holds 
\begin{equation}\label{eq:rate}
\fint_{B_r(x)}\abs{\nabla v_i\cdot\nabla v_j-\delta_{ij}}\di\haus^N\le C(N)rf(x)\, ,\quad\text{for any $0<r<1$}\, ,
\end{equation} 
for some integrable function $f:B_5(p)\setminus E\to[0,\infty)$. The strategy is borrowed from \cite{BrueNaberSemola22}, it is based on a weighted maximal function argument and a telescopic estimate, building on top of the Poincar\'e inequality, and it heavily exploits the $L^2$-Hessian bounds for splitting maps. The small content bound \eqref{eq:content} allows the construction to be iterated on the bad balls $B_{r_i}(x_i)$ and the results to be summed up into a geometric series.

\smallskip

In order to control the approximating measure $\mu_r$ on almost Euclidean balls, it is enough to plug \eqref{eq:rate} into the quantitative volume bound \eqref{eq:volume conv intro}.

To prove that the metric measure boundary is vanishing we need to show that at $\haus^N$-almost any point it is possible to slightly perturb the map $v$ above so that, morally, $f(x)=0$.
To this aim we perturb the splitting function $v$ at the second order so that, roughly speaking, it has vanishing Hessian at a fixed point $x$. The idea is to use a quadratic polynomial in the components of $v$ to make the second order terms in the Taylor expansion of $v$ at $x$ vanish. However its implementation is technically demanding and it requires the second order differential calculus on $\RCD$ spaces developed in \cite{Gigli18}. The construction is of independent interest and it plays the role of \cite[Lemma 6.2]{LytchakKapovitchPetrunin21} (see also \cite{Perelman95}) in the present setting. 
%Let us stress that $\delta$-splitting maps played a key technical role in the theory of smooth Riemannian manifolds with Ricci bounded below and Ricci limit spaces, since the pioneering work of Colding \cite{Colding97} and Cheeger-Colding \cite{CheegerColding96, CheegerColding97} and more recently in \cite{CheegerNaber15,CheegerJiangNaber21}. 
%{\color{red}Subsequent breakthroughs in the theory have been typically associated to constructions of ``better behaved $\delta$-splitting maps'': for instance a pivotal technical tool introduced in \cite{CheegerNaber15} and further pushed in \cite{CheegerJiangNaber21} (see also \cite{BrueNaberSemola22} for the case of $\RCD$ spaces) is the so-called ``transformation theorem'' which allows to effectively control the degeneration of $\delta$-splitting maps at singular points.\footnote{DS: rivedere la frase precedente per evitare ira di Naber} The construction of $\delta$-splitting maps with vanishing hessian at a given regular point performed in this work (see Section \ref{subsec:second order}) should be contextualised in such a toolkit of ``constructing well behaved $\delta$-splitting maps''. Since the previous constructions have been extremely fruitful, it is natural to expect that also our new improvement will be useful in the future, in addition to the applications already established in this paper.  
%}

\subsection*{Acknowledgements}
The first named author is supported by  the Giorgio and Elena Petronio Fellowship at the Institute for Advanced Study. The second and the last named authors are supported by the European Research Council (ERC), under the European Union Horizon 2020 research and innovation programme, via the ERC Starting Grant  “CURVATURE”, grant agreement No. 802689.\\ 
The last author is grateful to Alexander Lytchak for inspiring correspondence about the topics of the present work.

\section{Metric measure boundary and regular balls}
The aim of this section is to prove \autoref{thm:mms0} by assuming the following $\eps$-regularity theorem. The latter provides effective controls on the boundary measure for \textit{regular balls}. Here and in the following, we say that a ball $B_r(p)$ of an $\RCD(-(N-1),N)$ space is $\delta$-regular if 
\begin{equation}
	\dist_{GH}\left(B_{r}(p),B_{r}^{\setR^N}(0)\right) \le \delta r \, .
\end{equation}

\begin{theorem}[$\eps$-regularity]
	\label{prop:boundary on regular balls}
 For every $\eps>0$ and $N\in \mathbb{N}_{\geq 1}$, there exists $\delta(N,\eps)>0$ such that for all $\delta<\delta(N,\eps)$ the following holds.
	If $(X,\dist,\haus^N)$ is an $\RCD(-\delta (N-1),N)$ space, $p\in X$, and $B_{10}(p)$ is $\delta$-regular, then
	\begin{equation}\label{eq:uniformboundsep}
		|\mu_r|(B_1(p)) \le \eps \, , \quad
		\text{for any $r\in (0,1)$} \, .
	\end{equation}
	Moreover, $\abs{\mu_r}(B_1(p))\to 0$ as $r\downarrow 0$.
\end{theorem}

\subsection{Proof of \autoref{thm:mms0}}

We combine the $\eps$-regularity result \autoref{prop:boundary on regular balls} with the quantitative covering argument \cite[Theorem 5.2]{BrueNaberSemola22}. We also refer the reader to the previous works \cite{JiangNaber16,CheegerJiangNaber21} where this type of quantitative covering arguments originate from, and to \cite{LiNaber,LytchakKapovitchPetrunin21} for similar results in the setting of Alexandrov spaces. 
\smallskip

We recall that $B_{r}(p)$ is said to be a $\eta$-boundary ball provided
\begin{equation}
	\dist_{GH}\left(B_r(p), B_r^{\setR^N_+}(0)\right)\le \eta r \, ,
\end{equation}
where we denoted by $\setR^N_+$ the Euclidean half-space of dimension $N$ with canonical metric.

\begin{theorem}[Boundary-Interior decomposition theorem]\label{thm:decompositiontheorem}
	For any $\eta>0$ and $\RCD(-(N-1),N)$ m.m.s. $(X,\dist,\haus^N)$ with $p\in X$ such that $\haus^N(B_1(p))\ge v$, there exists a decomposition
	\begin{equation}
		B_1(p)\subset\bigcup_{a}B_{r_a}(x_a)\cup\bigcup_{b}B_{r_b}(x_b)\cup\tilde{\mathcal{S}}\, ,
	\end{equation}
	such that the following hold:
	\begin{itemize}
		\item[i)] the balls $B_{20r_a}(x_a)$ are $\eta$-boundary balls and $r_a^2\le \eta$;
		\item[ii)] the balls $B_{20r_b}(x_b)$ are $\eta$-regular and $r_b^2\le \eta$;
		\item[iii)]  $\haus^{N-1}(\tilde{\mathcal{S}})=0$;
		\item[iv)] $\sum_br_b^{N-1}\le C(N,v,\eta)$;
		\item[v)] $\sum_ar_a^{N-1}\le C(N,v)$.
	\end{itemize}
\end{theorem}

We point out that the statement of \autoref{thm:decompositiontheorem} is slightly different from the original one in \cite[Theorem 5.2]{BrueNaberSemola22} as we claim that the balls $B_{20r_b}(x_b)$ are $\eta$-regular, rather than considering the balls $B_{2r_b}(x_b)$. This minor variant follows from the very same strategy.

\medskip

Let us now prove the effective bound \eqref{eq:bound effective}.
Fix $\eta<1/4$. We apply \autoref{thm:decompositiontheorem} to find the cover
\begin{equation}
	B_1(p)\subset \bigcup_bB_{r_b}(x_b)\cup \tilde{\mathcal{S}}\, ,
\end{equation}
where $\haus^{N-1}(\tilde{\mathcal{S}})=0$, the balls $B_{20r_b}(x_b)$ are $\eta$-regular with $r_b^2\le \eta$ for any $b$, and 
\begin{equation}\label{eq:rdN-1C}
	\sum_br_b^{N-1}\le C(N,v,\eta)\, .
\end{equation}
Notice that boundary balls do not appear in the decomposition as we are assuming that
	\begin{equation}\label{z18}
	\partial X \cap B_2(p)=\emptyset \, ,
	\qquad
	\eta<1/4 \, .
	\end{equation}
	Indeed, any boundary ball intersects $\partial X$ as a consequence of \cite[Theorem 1.2]{BrueNaberSemola22}. Hence if a boundary ball appears in the decomposition, then  $B_{r_a}(x_a)\cap B_1(p)\neq \emptyset $ and $B_{r_a}(x_a)\subset B_2(p)$, contradicting \eqref{z18}.

\medskip

We fix $\eps=1/10$ and choose $\eta: = \delta(N,1/10)$ given by \autoref{prop:boundary on regular balls}. Then we estimate 
\begin{equation}\label{eq:cov}
	\abs{\mu_r}(B_1(p))\le \sum_b \abs{\mu_r}(B_{r_b}(x_b))\, ,
\end{equation}
by distinguishing two cases: if $r<r_b$, then the scale invariant version of \autoref{prop:boundary on regular balls} applies yielding
\begin{equation}\label{eq:r<}
	\abs{\mu_r}(B_{r_b}(x_b))\le \frac{1}{10}r_b^{N-1}\, .
\end{equation}
If $r>r_b$, then it is elementary to estimate
\begin{equation}\label{eq:r>}
	\begin{split}
		\abs{\mu_r}(B_{r_b}(x_b))
		\le &\int_{B_{r_b}(x_b)}\frac{1}{r}\abs{1-\frac{\haus^N(B_r(x))}{\omega_Nr^N}}\di\haus^N(x)\\
		\le & C(N,v)\frac{r_b^N}{r}\\
		\le &C(N,v)r_b^{N-1}\, .
	\end{split}
\end{equation}
The combination of \eqref{eq:rdN-1C}, \eqref{eq:cov}, \eqref{eq:r<} and \eqref{eq:r>} shows that 
\begin{equation}\label{eq:finite}
	\abs{\mu_r}(B_1(p))\le C(N,v)\sum_br_b^{N-1}\le C(N,v)\, ,
\end{equation}
as we claimed.

\medskip

	We finally prove that $|\mu_r|(B_1(p))\to 0$ as $r\downarrow 0$ by employing (the scaling invariant version of) \eqref{eq:finite}. We appeal once more to the covering $\{B_{r_b}(x_b)\}_{b\in \mathbb{N}}$. For any $M>1$ we write
	\begin{equation}\label{z1}
		|\mu_r|(B_1(p)) \le \sum_{b\le M} |\mu_r|(B_{r_b}(x_b)) + \sum_{b > M} |\mu_r|(B_{r_b}(x_b)) \, .
	\end{equation}
	Thanks to \eqref{eq:finite}, we can estimate
	\begin{equation}\label{z2}
		\sum_{b > M} |\mu_r|(B_{r_b}(x_b))
		\le C(N,v) \sum_{b>M} r_b^{N-1} \, .
	\end{equation}
	By using that $|\mu_r(B_t(x))|\to 0$ as $r\downarrow 0$ when $B_t(x)$ is a $\delta(N,1/10)$-regular ball (see \autoref{prop:boundary on regular balls}), we get
	\begin{equation}\label{z3}
		\lim_{r\downarrow 0} \sum_{b\le M} |\mu_r|(B_{r_b}(x_b)) = 0 \, .
	\end{equation}
	The sought conclusion follows by combining \eqref{z1}, \eqref{z2}, \eqref{z3} and sending $M\to \infty$.

\section{Quantitative volume convergence via splitting maps}

It is a classical fact \cite{CheegerColding96,Colding97,Cheeger01,DePhilippisGigli18,BrueNaberSemola22} that for any $\eps>0$ there exists $\delta=\delta(\eps,N)>0$ such that the following holds: If $(X,\dist,\meas)$ is an $\RCD(-\delta^2(N-1),N)$ metric measure space and $u:B_{10}(p)\to\setR^N$ is a $\delta^2$-splitting map (see \autoref{def:deltasplitting} below), then
\begin{equation}\label{eq:volume bound}
	\dist_{GH}\left(B_1(p),B_1^{\setR^N}(0)\right) <\eps\, ,\quad\abs{\haus^N(B_1(p))-\omega_N} < \eps\, .
\end{equation}
The main result of this section is a quantitative version of \eqref{eq:volume bound} where $\eps$ is estimated explicitly in terms of $C(N)$ and a power of $\delta$. Before stating it, we recall the definition of $\delta$-splitting map and we introduce the relevant terminology.

\medskip

Given an $\RCD(-(N-1),N)$ space $(X,\dist,\haus^N)$, $p\in X$, and a harmonic map $u:B_{10}(p) \to \setR^N$, we define $\mathcal{E}:B_{10}(p)\to [0,\infty)$ by
\begin{equation}\label{eq:defE}
	\mathcal{E}(x):= \sum_{i, j} |\nabla u_i(x) \cdot \nabla u_j(x) - \delta_{ij}| \, .
\end{equation}

\begin{definition}\label{def:deltasplitting}
	Let $(X,\dist,\haus^N)$ be an $\RCD(-(N-1),N)$ metric measure space and fix $p\in X$. We say that a harmonic map $u:B_{10}(p)\to \setR^N$ is a $\delta$-splitting map provided
	\begin{equation}\label{eq:defsplitting}
		\delta^2:=\fint_{B_{10}(p)} \mathcal{E} \di \haus^N  \le  1 \, .
	\end{equation}
\end{definition}

We refer to \cite{CheegerNaber15,CheegerJiangNaber21,BrueNaberSemola22} for related results about harmonic splitting maps on spaces with lower Ricci bounds.

	\begin{proposition}[Quantitative volume convergence]
	\label{prop:quantitativevolcE}
	For every $N\in \mathbb{N}_{\geq 1}$, there exists a constant $C(N)>0$ such that the following holds. Let $(X,\dist,\haus^N)$ be an $\RCD(-(N-1),N)$ metric measure space. Let $p\in X$ and let $u:B_{10}(p)\to\setR^N$ be a harmonic $\delta$-splitting map, for some $\delta\le \delta(N)$. Then for any $x\in B_4(p)$ it holds 
	\begin{equation}\label{eq:integratedboundpr}
		\abs{1-
			\frac{\haus^N(B_r(x))}{\omega_N r^N}}
		\le C(N)\left( r^2
		+\int_0^r\fint_{B_{4t}(x)}\mathcal{E}\di\haus^N\frac{\di t}{t}
		\right)\, ,
	\end{equation}
	for any $0<r<1$.
\end{proposition}

\begin{remark}
In particular, the above is a quantitative version of the classical volume $\eps$-regularity theorem for almost Euclidean balls, where the closeness to the Euclidean model can be quantified in terms of the \emph{best} splitting map $u:B_4(p)\to\setR^N$.
\end{remark}

The elliptic regularity for harmonic functions on $\RCD(-(N-1),N)$ spaces guarantees that any $\delta$-splitting map $u:B_{10}(p)\to\setR^N$ as above is $C(N)$-Lipschitz on $B_{9}(p)$.

Moreover the map $u$ satisfies the following sharp Lipschitz bound and $L^2$-Hessian bound:
\begin{align}
 |\nabla u|^2\le 1+ C(N)\delta^2 \; \text{ in $B_9(p)$}; \label{eq:(i)}\\
\fint_{B_9(p)} |\Hess u|^2\di \haus^N \le C(N) \delta^2. \label{eq:(ii)}
\end{align}
We refer to \cite[Lemma 4.3]{HondaPeng22} for the proof of the sharp Lipschitz bound with a variant of an argument originating in \cite{CheegerNaber15}. The $L^2$-Hessian bound can be easily obtained integrating  the Bochner's inequality with Hessian term against a good cut-off function and employing \eqref{eq:defsplitting}; this argument originated in \cite{CheegerColding96} (see also \cite{BruePasqualettoSemola19} for the implementation in the $\RCD$ setting). We refer to \cite{Gigli18} for the relevant terminology and background about second order calculus on $\RCD$ spaces.

\begin{remark}
	More in general, if $(X,\dist,\meas)$ is an $\RCD(-\delta^2(N-1),N)$ space, and there exists a $\delta$-splitting map $u:B_{10}(p)\to \setR^N$, then $\meas=\haus^N$ up to multiplicative constants, i.e. the space is noncollapsed, see \cite{DePhilippisGigli18,Honda19,BrenaGigliHondaZhu21}.
\end{remark}

\subsection{Volume convergence and approximate distance}
	
We fix an $\RCD(-(N-1),N)$ space $(X,\dist,\haus^N)$ and a point $x\in X$.
We consider a $C(N)$-Lipschitz function
\begin{equation}
	r:B_2(x)\to [0,\infty) \, ,
	\qquad
	r(0) = 0 \, ,
\end{equation} 
belonging to the domain of the Laplacian on $B_2(x)$. It is well-known that, if the lower Ricci bound is reinforced to nonnegative Ricci curvature, and
\begin{equation}\label{eq:functcone}
	\Delta r^2 = 2N \, ,
	\quad
	|\nabla r|^2 = 1\, , \quad\text{on $B_1(p)$}\, ,
\end{equation}
then $B_1(x)$ is a metric cone and $r(x)=\dist(x,p)$, see \cite{CheegerColding96,DePhilippisGigli16}. In particular
\begin{equation}
	t \to \frac{\haus^N(B_t(x))}{\omega_N t^N}\, ,
	\qquad
	\text{is constant in $(0,1)$} \, .
\end{equation}
The next result provides a quantitative control on the derivative of the volume ratio in terms of suitable norms of the error terms in \eqref{eq:functcone}, $|\nabla r|^2-1$ and $\Delta r^2 - 2N$.

\begin{lemma}\label{lemma:derv}
	For $\Leb^1$-a.e. $0<t<1$ it holds
	\begin{align}
		\nonumber   -\frac{\di}{\di t}\frac{\haus^N(B_t(x))}{\omega_Nt^N}\le& \, \frac{C(N)}{t}\fint_{B_t(x)}\abs{\Delta r^2-2N}\di\haus^N+\frac{C(N)}{t}\fint_{\partial B_t(x)}\abs{\abs{\nabla r}^2-1}\di\haus^{N-1}\\
		&\, +\frac{C(N)}{t^2}\left(\fint_{\partial B_t(x)}r\di\haus^{N-1}-t\right)\, .\label{der:mainlemma}
	\end{align}
	In particular, if $x\in X$ is a regular point, it holds that for every $t\in (0,1)$,
	\begin{align}
		\nonumber 1-\frac{\haus^N(B_t(x))}{\omega_Nt^N}\le&\, C(N)\int_0^t\left[\fint_{B_s(x)}\abs{\Delta r^2-2N}\di\haus^N+\fint_{\partial B_s(x)}\abs{\abs{\nabla r}^2-1}\di\haus^{N-1}\right]\frac{\di s}{s}\\
		&\, +\int_0^t\left(\fint_{\partial B_s(x)}r\di\haus^{N-1}-s\right)\frac{\di s}{s^2}\, .\label{eq:integratedgeneral}
	\end{align} 
\end{lemma}

\begin{proof}	
For any $x\in X$, the function
\begin{equation*}
	t\mapsto \frac{\haus^N(B_t(x))}{t^N}
\end{equation*}
is locally Lipschitz and differentiable at every $t\in (0,\infty)$. Moreover,
\begin{equation}\label{eq:derBG}
	\frac{\di }{\di t} \frac{\haus^N(B_t(x))}{t^N}=\frac{\haus^{N-1}(\partial B_t(x))}{t^N}-\frac{N\haus^N(B_t(x))}{t^{N+1}}\, ,
\end{equation}
for a.e. $t\in (0,\infty)$, as a consequence of the coarea formula. We also notice that $\haus^{N-1}(\partial B_t(x))=\Per(B_t(x))$ for a.e. $t\in (0,\infty)$.

\medskip
	
For a.e. $t\in (0,1)$ it holds
	\begin{align}
		\nonumber	\int_{\partial B_t(x)} r |\nabla (\dist_x - r)^2| \di \haus^N
			& = 
			\int_{\partial B_t(x)} r(1+|\nabla r|^2) \di \haus^{N-1} - 2\int_{\partial B_t(x)} r \nabla r \cdot \nabla \dist_x \di \haus^{N-1}
	\label{eq:splitest1}		\\& =: I + II\,,
	\end{align}
    since $\haus^N$-a.e. on $X$ it holds that $\abs{\nabla \dist_x}=1$.

    Let us estimate $I$ in \eqref{eq:splitest1}:
    \begin{align}
    \nonumber			I   \le &\, 2\int_{\partial B_t(x)} r \di \haus^{N-1} + \int_{\partial B_t(x)}r||\nabla r|^2 - 1| \di \haus^{N-1}
    			\\ \le &\,  2\int_{\partial B_t(x)} r \di \haus^{N-1} \nonumber
  \nonumber  			+ C(N) t \int_{\partial B_t(x)} \abs{\abs{\nabla r}^2-1} \di \haus^{N-1} \\
\nonumber			\le \,  &2t\haus^{N-1}(\partial B_t(x))+2\abs{\int_{\partial B_t(x)}r\di\haus^{N-1}-t\haus^{N-1}(\partial B_t(x))}\\
&\, +C(N)t\int_{\partial B_t(x)} \abs{\abs{\nabla r}^2-1} \di \haus^{N-1} \, .\label{eq:boundI2}
    \end{align}
    
  Away from a further $\Leb^1$-negligible set of radii $t\in (0,1)$, we can estimate $II$ with the Gauss-Green formula from \cite{BruePasqualettoSemola19}, after recalling that the exterior unit normal of $B_t(x)$ coincides $\haus^{N-1}$-a.e. with $\nabla \dist_x$, see \cite[Proposition 6.1]{BruePasqualettoSemola21}. We obtain
    \begin{align}
   \nonumber 	-2\int_{\partial B_t(x)} r \nabla r \cdot \nabla \dist_x \di \haus^{N-1} =& -\int_{B_t(x)} \Delta r^2 \di \haus^N\\
    	\le & -2N \haus^N(B_t(x)) +\abs{\int_{B_t(x)}\left(\Delta r^2-2N\right)\di\haus^N}  \, .\label{eq:gaplapla}
    \end{align}
   The combination of \eqref{eq:boundI2} and \eqref{eq:gaplapla} together with \eqref	{eq:splitest1} proves that  
   \begin{align*}
   \int_{\partial B_t(x)} r |\nabla (\dist_x - r)^2| \di \haus^N\le & 2t\haus^{N-1}(\partial B_t(x))-2N \haus^N(B_t(x))\\
   &+2\abs{\int_{\partial B_t(x)}r\di\haus^{N-1}-t\haus^{N-1}(\partial B_t(x))}\\
   &+ C(N)t\int_{\partial B_t(x)} \abs{\abs{\nabla r}^2-1} \di \haus^{N-1}\\
   &+\abs{\int_{B_t(x)}\left(\Delta r^2-2N\right)\di\haus^N} \, .
   \end{align*}
   Hence, combining \eqref{eq:derBG} with last estimate, we get 
   \begin{align*}
\nonumber   -\frac{\di}{\di t}\frac{\haus^N(B_t(x))}{\omega_Nt^N}\le& \frac{C_N}{t}\fint_{B_t(x)}\abs{\Delta r^2-2N}\di\haus^N+\frac{C_N}{t}\fint_{\partial B_t(x)}\abs{\abs{\nabla r}^2-1}\di\haus^{N-1}\\
   &+\frac{C_N}{t^2}\abs{\fint_{\partial B_t(x)}r\di\haus^{N-1}-t}\, .
   \end{align*}
 The second conclusion in the statement follows by integrating the first one, as the function
 \begin{equation*}
 s\mapsto \frac{\haus^N(B_s(x))}{\omega_Ns^N}
 \end{equation*}  
 is locally Lipschitz and limits to $1$ as $s\downarrow 0$.
\end{proof}

\subsection{Proof of \autoref{prop:quantitativevolcE}}

Let $u:B_{10}(p) \to \setR^N$ be a $\delta$-splitting map.
Fix $x\in B_2(p)$. Up to the addition of some constant that do not affect the forthcoming statements, we can assume that $u(x) =0$ and define
\begin{equation}
	r^2 = \sum_i u_i^2 \, .
\end{equation}

We estimate the gap between $r$ and the distance function from $x$ and between $\Delta r^2$ and $2N$ in terms of the quantity $\mathcal{E}$ introduced in \eqref{eq:defsplitting}.

\begin{lemma}\label{lemma:bound r}
	The following inequalities hold $\haus^N$-a.e. in $B_8(p)$:
	\begin{equation}\label{eq:r_lip}
		r \le (1+ C(N)\delta^2)\dist_x
	  \, ,
	\end{equation}
	\begin{equation}\label{eq:est E}
		||\nabla r|^2 - 1| +	|\Delta r^2 - 2N| \le 3 \mathcal{E} \, .
	\end{equation}
\end{lemma}

\begin{proof}
	The first conclusion follows from the sharp Lipschitz bound \eqref{eq:(i)}. Indeed,
	\begin{equation*}
		r(y)^2 
		= \sum_i |u_i(y) - u_i(x)|^2
		= |u(y) - u(x)|^2 \le (1+C(N)\delta^2)^2 \dist(x,y)^2\, .
	\end{equation*}
	To show \eqref{eq:est E} we employ the identities
	\begin{equation*}
		|\nabla r^2|^2 =4 \sum_{i,j} u_i u_j \nabla u_i \cdot \nabla u_j
		\, ,
	\end{equation*}
	\begin{equation*}
		\Delta r^2 =\sum_{i}\Delta u_{i}^2= 2\sum_i |\nabla u_i|^2 \, ,
	\end{equation*} 
	that can be obtained via the chain rule taking into account that $\Delta u=0$, together with some elementary algebraic manipulations.
\end{proof}

\begin{corollary}\label{cor:gapBGharm}
Under the same assumptions and with the same notation above, for $\Leb^1$-a.e. $0<r<1$ it holds that 
     \begin{equation}\label{eq:derestcor}
     	- \frac{\di}{\di t} \left( \frac{\haus^{N}(B_t(x))}{t^N}   \right)
     	\le 
     	C(N)t^{-N} \int_{\partial B_t(x)} \mathcal{E} \di \haus^{N-1} + C(N) t^{-N - 1}\int_{B_{4t}(x)} \mathcal{E} \di \haus^N \, .
     \end{equation}
Moreover, if $x\in X$ is regular, then for any $0<r<1$ it holds
\begin{equation}\label{eq:integratedbound}
1-\frac{\haus^N(B_r(x))}{\omega_Nr^N}
\le 
C(N)  \int_0^r\fint_{B_{4t}(x)}\mathcal{E}\di\haus^N\frac{\di t}{t}\, .
\end{equation}

\end{corollary}

\begin{proof}
In order to prove \eqref{eq:derestcor}, it is sufficient to employ \eqref{der:mainlemma} in combination with \autoref{lemma:bound r}. Indeed, the bound for the first two summands at the right hand side of \eqref{der:mainlemma} follows directly from \eqref{eq:est E}.
In order to bound the last summand we notice that, thanks to the sharp Lipschitz estimate  \eqref{eq:(i)} applied on the ball $B_{4t}(x)$,
    \begin{equation*}
    	r(y) = |u(y) - u(x)| \le \norm{\nabla u}_{L^\infty(B_{2t}(x))} \dist(x,y) \le \left(1+C(N)\fint_{B_{4t}(x)}\mathcal{E}\di \haus^{N}\right)\dist(x,y)\, ,
    \end{equation*}
    for any $y\in B_t(x)$, hence
    \begin{equation}\label{eq:bound1}
    	\int_{\partial B_t(x)} r \di \haus^{N-1}
    	\le t\haus^{N-1}(\partial B_t(x)) + C(N) \int_{B_{4t}(x)} \mathcal{E} \di \haus^N \, .
    \end{equation}
\medskip

The estimate \eqref{eq:integratedbound} follows by integrating \eqref{eq:derestcor} in $t$. Indeed, integrating by parts in $t$ and using the coarea formula, we obtain
\begin{align}
\int_{0}^{r}t^{-N} \int_{\partial B_{t}(x)} \mathcal{E} \di \haus^{N-1}  \di t  &\leq  C(N) \Big( \fint_{B_{r}(x)}  \mathcal{E} \di \haus^N - \liminf_{t\downarrow 0}  \fint_{B_{t}(x)}  \mathcal{E} \di \haus^N \nonumber \\
& \qquad \qquad \quad + \int_{0}^{r}  \fint_{B_{t}(x)} \mathcal{E} \di \haus^N \frac{\di t}{t} \Big)\nonumber \\
&\leq   C(N)\Big(  \fint_{B_{r}(x)}  \mathcal{E} \di \haus^N + \int_{0}^{r}  \fint_{B_{t}(x)} \mathcal{E} \di \haus^N \frac{\di t}{t}\Big)\,. \label{eq:UBIntBP}
\end{align}
By using the Bishop-Gromov inequality, we estimate
\begin{align} 
	\fint_{B_{r}(x)}\mathcal{E}\di\haus^N
	 & \le
	 C(N)\int_{r/4}^r \fint_{B_{4t}(x)} \mathcal{E} \di \haus^N \frac{\di t}{t} \nonumber
	 \\& \le
	  C(N)\int_0^{r}\fint_{B_{4t}(x)}\mathcal{E}\di\haus^N\frac{\di t}{t} \, .	 \label{eq:UBIntBPBG}	
    \end{align}
The claimed bound \eqref{eq:integratedbound} follows then by integrating \eqref{eq:derestcor} in $t$, taking into account \eqref{eq:UBIntBP} and \eqref{eq:UBIntBPBG}.
\end{proof}

Given \autoref{cor:gapBGharm}, to conclude the proof of \autoref{prop:quantitativevolcE} it is enough to estimate the negative part of
\begin{equation*}
	1-\frac{\haus^N(B_r(x))}{\omega_N r^N}\, . 
\end{equation*}
This goal can be easily achieved using the Bishop-Gromov inequality, arguing as in \cite{LytchakKapovitchPetrunin21}:
\begin{equation}\label{eq:lower}
	\begin{split}
		1-\frac{\haus^N(B_r(x))}{\omega_N r^N} 
		&\ge 1- \frac{\haus^N(B_r(x))}{v_{-1,N}(r)} + \frac{\haus^N(B_r(x))}{v_{-1,N}(r)} - \frac{\haus^N(B_r(x))}{\omega_N r^N}
		\\\ge & - \frac{\haus^N(B_r(x))}{v_{-1,N}(r)}\abs{\frac{v_{-1,N}(r) - \omega_N r^N}{\omega_N r^N}}\, ,
	\end{split}	
\end{equation}
where $v_{-1,N}(r)$ is the volume of the ball of radius $r$ in the model space of constant sectional curvature $-1$ and dimension $N\in \setN$.\\
Using the well-known expansion of $v_{-1,N}(r)$ around $0$ and the Bishop-Gromov inequality, we deduce that
\begin{equation}\label{eq:lower2}
	\frac{\haus^N(B_r(x))}{v_{-1,N}(r)}\abs{\frac{v_{-1,N}(r) - \omega_N r^N}{\omega_N r^N}}
	\le C(N) r^2 \, \quad
	\text{for any $r\in (0,10)$} \, .
\end{equation}
In conclusion, the combination of \eqref{eq:integratedbound}  with \eqref{eq:lower} and  \eqref{eq:lower2}  proves the following:
\begin{equation}\label{eq:integratedbound2}
	\abs{1-\frac{\haus^N(B_r(x))}{\omega_Nr^N}}
	\le C(N)\left( r^2+\int_0^r\fint_{B_{4t}(x)}\mathcal{E}\di\haus^N\frac{\di t}{t}\right)\, .
\end{equation}

\section{Second order corrections}\label{sec:secondordercorr}

In order to prove that the metric measure boundary is vanishing directly through the quantitative volume estimate (cf. \autoref{prop:quantitativevolcE}) we would need to build $\delta$-splitting maps whose Hessian is zero, in suitable sense, on a big set. 
This seems to be definitely hopeless, even on general smooth Riemannian manifolds, as the gradients of the components of the splitting map would be parallel vector fields. In order to overcome this issue we argue as follows:
\begin{itemize}
	\item[(i)] First, we show that the metric measure boundary is absolutely continuous with respect to $\haus^N$; this is done in \autoref{sec:control of the mmb} below.
	\item[(ii)] In a second step, we show that the density of the boundary measure with respect to $\haus^N$ is zero almost everywhere; this will be an outcome of  \autoref{prop:or} in this section.
\end{itemize}
We remark that after establishing (i), the vanishing of the metric measure boundary for Alexandrov spaces with empty boundary would follow directly from \cite[Theorem 1.7]{LytchakKapovitchPetrunin21}.

In order to prove (ii), a key step is to build maps whose Hessian vanishes at a fixed point. In order to do so, we will allow for some extra flexibility on the $\delta$-splitting map.
More precisely, for $\haus^N$ a.e. $x\in X$ we can build an \textit{almost $\delta$-splitting} $u: B_{10}(p)\to \setR^N$, meaning that $\Delta u$ is not necessarily zero in a neighbourhood of $x$ but rather converging to $0$ at sufficiently fast rate at $x$, that satisfies 
\begin{equation}
	\lim_{r\downarrow 0}\fint_{B_r(x)}|\Hess u|^2 \di \haus^N = 0 \, .
\end{equation}
The construction of these maps is of independent interest and pursued in \autoref{subsec:second order}.

\begin{proposition}\label{prop:or}
	Let $(X,\dist,\haus^N)$ be an $\RCD(-(N-1),N)$ metric measure space. Then for $\haus^N$-a.e. $x\in X$ it holds
	\begin{equation}\label{eq:or}
		\lim_{r\downarrow 0}\frac{1}{r}\abs{1-\frac{\haus^N(B_r(x))}{\omega_Nr^N}}=0\, .
	\end{equation}
\end{proposition}

\begin{remark}	
	The volume convergence and the classical regularity theory imply that
	\begin{equation}
		\lim_{r\downarrow 0}\abs{1-\frac{\haus^N(B_r(x))}{\omega_Nr^N}}=0\,, \quad \text{ for $\haus^N$-a.e. $x\in X$}. 
	\end{equation}
	An improved  convergence rate $o(r^{\alpha})$, for some $\alpha=\alpha(N)<1$ should follow from the arguments in \cite{Colding97, CheegerColding97,CheegerColding2000b} for noncollapsed Ricci limit spaces. More precisely, in \cite[Section 3]{CheegerColding2000b}, it was explicitly observed that one can obtain a rate of convergence for the scale invariant Gromov-Hausdorff distance between balls $B_r(x)$ and Euclidean balls on a set of full measure. It seems conceivable that, along those lines, one can also obtain 
\begin{equation}	
	\lim_{r\downarrow 0} \frac{1}{r^\alpha}\abs{1-\frac{\haus^N(B_r(x))}{\omega_Nr^N}}=0\, , 
\end{equation}	
	for some $\alpha=\alpha(N)<1$, for $\haus^N$-a.e. $x$.	 
However, to the best of our knowledge, the existence of a single point where the $o(r)$ volume convergence rate \eqref{eq:or} holds is new even for Ricci limit spaces.
	This improvement will play a pivotal role in the analysis of the metric measure boundary.
\end{remark}

The proof of \autoref{prop:or} is based on a series of auxiliary results and it is postponed to the end of the section. The strategy is to apply \autoref{lemma:derv} to a different function $r$ defined out of the new coordinates built in \autoref{subsec:second order}. In \autoref{subsec:volume via almost splitting} we check that $r$, $|\nabla r|^2 - 1$ and $\Delta r^2 - 2N$ enjoy all the needed asymptotic estimates.

\subsection{$\delta$-splitting maps with vanishing Hessian at a reference point}
\label{subsec:second order}

The almost $\delta$-splitting map with vanishing Hessian at a given point is built in \autoref{lemma:algebraicomp}. The key step in the construction is provided by \autoref{prop:goodHessian} below.

\begin{proposition}\label{prop:goodHessian}
For any $\eps>0$, if $\delta \le \delta(N,\eps)$ the following property holds.
Given an $\RCD(-(N-1),N)$ m.m.s  $(X,\dist,\haus^N)$, $p\in X$, and a harmonic $\delta$-splitting map $u:B_{10}(p)\to\setR^N$ there exists a set $E\subset B_1(p)$ such that the following hold:
\begin{itemize}
\item[(i)] $\haus^N(B_1(p)\setminus E)\le \eps$;
\item[(ii)] for any $x\in E$ there exists an $N\times N$ matrix $A_x$ such that $\abs{A_x-\Id}\le \eps$ and the map $u^x:=A_x\circ u:B_{10}(p)\to \setR^N$ verifies 
\begin{itemize}
\item[(a)] $\nabla u^x_i(x)\cdot\nabla u^x_j(x)=\delta_{ij}$ and $x$ is a Lebesgue point of $\nabla u^x_i(\cdot)\cdot\nabla u^x_j(\cdot),$ for all $i, j=1,\dots, N$, i.e.
\begin{equation}\label{eq:Lebnablauiuj}
 \lim_{r\downarrow 0}\fint_{B_r(x)}\abs{\nabla u^x_i\cdot \nabla u^x_j-\delta_{ij}}\di\haus^N=0\, .
 \end{equation}
\item[(b)] $\fint _{B_r(x)}\abs{\Hess u^x}\di\haus^N\le \eps$ for any $0<r<1$;
\end{itemize}
\item[(iii)] for any $x\in E$ and for any $k=1,\dots, N$ there exist coefficients $\alpha^k_{ij}$ with $\alpha^k_{ij}=\alpha^k_{ji}$ for any $i,j,k$, such that it holds
\begin{equation}\label{eq:vansihHess}
\lim_{r\downarrow 0}\fint_{B_r(x)}\abs{\Hess u^x_k+\sum_{i,j}\alpha^k_{ij}\nabla u^x_i\otimes\nabla u^x_j}^2\di\haus^N= 0\, .
\end{equation}
\end{itemize}
\end{proposition}

We state and prove an elementary lemma, \autoref{lemma:linear algebra}. It says that any almost orthogonal matrix $A\in \setR^{N\times N}$ becomes exactly orthogonal after multiplication with some $B\in \setR^{N\times N}$ which is close to the identity. This result will be applied to $A_{ij} := \nabla u_i(x) \cdot \nabla u_j(x)$ where $u:B_{10}(p)\to \setR^N$ is a $\delta$-splitting map and $x\in B_1(p)$ is a point where $|\Hess u(x)|$ is small in an appropriate sense. The matrix $B$ provided by \autoref{lemma:linear algebra} will be used to define a new $\delta$-splitting map $v:= B\circ u$ which is well normalized at $x$.

   \begin{lemma}\label{lemma:linear algebra}
  	For any $\delta \le \delta_0(N)$ the following property holds. For any $A\in \setR^{N\times N}$ satisfying
  	\begin{equation}
  		\abs{A\cdot A^t - I} \le \delta 
  	\end{equation}
  	there exists $B\in \setR^{N\times N}$ such that 
  	\begin{equation}
  		(BA)\cdot (BA)^t = I \, , \quad
  	    \abs{B-I}\le C(N)\delta\, .
  	\end{equation}
  \end{lemma}  
  
  \begin{proof}
  	It is enough to consider $B^{-1}= \sqrt{ A\cdot A^t}$, which is well-defined because $A\cdot A^t$ is symmetric, positive definite, and invertible provided $\delta\le \delta(N)$.\\
   Notice that the square root is $C(N)$-Lipschitz in a neighbourhood of the identity, hence
	\begin{equation}
	\abs{\sqrt{A\cdot A^t}-I}\le C(N)\abs{A\cdot A^t-I}
	=C(N)\delta\, .
	\end{equation}  
Analogously, the inversion is $C(N)$-Lipschitz in a neighbourhood of the identity, hence
\begin{equation}
\abs{B-I}\le C(N)\abs{\sqrt{A\cdot A^t}-I}\le C(N)\delta\, .
\end{equation}	
\end{proof}

\begin{proof}[Proof of \autoref{prop:goodHessian}]
First of all, since $u:B_{10}(p)\to \mathbb{R}^N$ is a $W^{1,2}$-Sobolev map, then $\haus^N$-a.e. $x$ is a Lebesgue point of $\abs{\nabla u_i}^2$, for all $i=1,\ldots, N$ and for $\nabla u_i\cdot \nabla u_j$ for any $i,j=1,\dots,N$.  Without further comments, the sets $\tilde{E}$ and $E$ constructed below will be assumed to be contained in such a set of full measure made of Lebesgue points of $\nabla u_i\cdot \nabla u_j$.
\smallskip

Let us fix $\delta<10^{-1}$ to be specified later in terms of $\eps$ and $N$. We set
\begin{equation}
	\tilde E:= \left\lbrace  x\in B_2(p) : \, \sup_{r<3} \fint_{B_r(x)} |\Hess u|^2 \di \haus^N \le \delta \right\rbrace \, .
\end{equation}
A standard maximal function argument, along with the estimate 
\begin{equation}
\fint_{B_{5}(p)} |\Hess u|^2\di \haus^N \le C(N)\delta^2\, ,
\end{equation}
implies that $\haus^N(B_1(p)\setminus \tilde E) \le C(N)\delta$.

Let us fix $x\in \tilde E$. The Poincar\'e inequality \cite{VonRenesse08,Rajala12} (cf. with the proof of \cite[Lemma 4.16]{BrueNaberSemola22}) gives
\begin{equation}
	\abs{\fint_{B_{2r}(x)} \mathcal{E}\di \haus^N - \fint_{B_{r}(x)} \mathcal{E}\di \haus^N }
	\le C(N) r \fint_{B_{3r}(x)} |\Hess u|\di \haus^N
	\le C(N) \delta^{1/2} r \, ,
\end{equation}
for any $r<1/2$. A standard telescopic argument implies that $x$ is a Lebesgue point for $\mathcal{E}$ and
\begin{equation}
	\mathcal{E}(x) \le \abs{\fint_{B_3(x)} \mathcal{E} \di \haus^{N} - \lim_{r\downarrow 0} \fint_{B_r(x)} \mathcal{E} \di \haus^N} + \fint_{B_3(x)} \mathcal{E} \di \haus^{N}
	\le C(N)(\delta^{1/2} + \delta^2) \, .
\end{equation}
If $\delta\le \delta(N)$ is small enough, we can apply \autoref{lemma:linear algebra} and find $A_x$ satisfying (ii)(a). To verify (ii)(b), we observe that
\begin{equation}
	|\Hess u^x| \le |A_x| |\Hess u| \le C(N) |\Hess u| \, .
\end{equation}
Let us finally prove (iii). First of all, the same telescopic argument as above gives
\begin{equation}\label{z6}
\sum_{i,j}	\fint_{B_r(x)} |\nabla u_i^x \cdot \nabla u_j^x - \delta_{ij}| \di \haus^N \le C(N) \delta^{1/2} r \, ,
\quad
\text{for any $r<1$} \, .
\end{equation}
We define $E$ as the set of those $x\in \tilde E$ satisfying the following properties:
\begin{itemize}
\item 
\begin{equation}\label{eq:(1)Hess}
	\lim_{r\downarrow 0} \frac{\haus^N(B_r(x) \setminus \tilde E)}{\haus^N(B_r(x))} = \lim_{r\downarrow 0} \frac{1}{r^{N}}\int_{B_r(x)\setminus \tilde E}  |\Hess u|^2 \di \haus^N = 0\,  ;
\end{equation}

\item there exist $\alpha_{ij}^k\in \setR$ such that
\begin{equation}\label{eq:(2)Hess}
	\lim_{r\downarrow 0} \fint_{B_r(x)} \abs{\Hess u_k^x(\nabla u_i^x, \nabla u_j^x) + \alpha_{ij}^k}^2 \di \haus^N = 0 \, ,
\end{equation}
for any $i,j,k=1, \ldots ,N$.
\end{itemize}
Observe that 
\begin{equation}\label{eq:alpha small}
	|\alpha_{ij}^k | \le C(N)\ , \qquad
	\text{for any $i,j,k=1, \ldots, N$}\, ,
\end{equation}
as a consequence of \eqref{eq:(2)Hess} and of the definition of $\tilde E$. Also, notice that $\haus^N(E\setminus \tilde E)=0$: 
it is obvious that \eqref{eq:(1)Hess} holds for $\haus^N$-a.e. $x\in \tilde E$; regarding \eqref{eq:(2)Hess}, we notice that it amounts to ask that $x$ is a Lebesgue point of $\Hess u(\nabla u_i,\nabla u_j)$ for any $i,j=1, \ldots, N$. Indeed, multiplying with $A_x$ does not change this property.

\smallskip

We now show (iii) for $\alpha_{ij}^k$ defined as in \eqref{eq:(2)Hess}.
Fix $x\in E$ and $\eta \ll \delta$.
Thanks to \eqref{z6}, \eqref{eq:(1)Hess} and \eqref{eq:(2)Hess}, we can find $r_0=r_0(\eta)\le 1$ such that for any $r<r_0$ there exists $G_r\subset B_r(x)$ satisfying
\begin{itemize}
	\item $G_{r}$ has $\eta$-almost full measure in $B_{r}(x)$, i.e.
	 \begin{equation}\label{eq:Greta}
	\haus^N(B_r(x)\setminus G_r)\le \eta \haus^N(B_r(x));
	\end{equation}
	\item $G_r \subset \tilde E$ and
	\begin{equation}\label{eq:Hesseta}
		\int_{B_r(x)\setminus \tilde E} |\Hess u|^2 \di \haus^N \le \eta \haus^N(B_r(x)) \, ;
	\end{equation}
	\item for any $y\in G_r$ it holds
	\begin{equation}\label{eq:gradeta}
		\sum_{i,j} |\nabla u_i^x(y) \cdot \nabla u_j^x(y) - \delta_{ij}|
		\le \eta \, 
	\end{equation}
	and
    \begin{equation}\label{eq:Hessalpha}
    	\abs{\Hess u_k^x(y)(\nabla u_i^x(y), \nabla u_j^x(y)) + \alpha_{ij}^k}
    	\le \eta \, ,
    \end{equation}
    for any $i,j,k=1, \ldots, N$.
\end{itemize}
In particular $A_{ij}:=\nabla u_i^x(y)\cdot \nabla u_j^x(y)$ is invertible for any $y\in G_r$.

Fix $r<r_0$.
We denote by $L^2(TX)$ the $L^\infty$-module of velocity fields over $X$, and by $L^2(TX\otimes TX)$ the $L^\infty$-module of $2$-tensors. We refer the reader to \cite{Gigli18} for the relevant background and terminology.

The identification of $L^2(TX)$ with the asymptotic GH-limits provided in \cite{GigliPasqualetto16b}, implies that the family $\{ \nabla u^x_i \, : \, i=1, \ldots, N\} \subset L^2(TX)$ is independent on $G_r$ (cf. \cite[Definition 1.4.1]{Gigli18}).
Using that $L^2(TX)$ has dimension $N$, see \cite{DePhilippisGigli18}, we infer that
\begin{equation}
	\mathcal{B}:= \left\lbrace  \nabla u_i^x\otimes \nabla u_j^x \, : \, i,j=1, \ldots, N \right\rbrace \subset L^2(TX\otimes TX)
\end{equation}
is a base of $L^2(TX\otimes TX)$ on $G_r$, according to \cite[Definition 1.4.3]{Gigli18} (see also \cite[Lemma 2.1]{BruePasqualettoSemola21}). In particular, there exists a family of measurable functions $\{f_{i,j}^k\}_{i,j,k}$ such that
\begin{equation}\label{eq:base Gr}
	\sum_{i,j} f^k_{ij} \nabla u^x_i \otimes \nabla u^x_j = \Hess u^x_k \in L^2(TX \otimes TX)
	\qquad
	\text{on $G_r$} \, ,
\end{equation}
see the discussion in \cite[Page 36]{Gigli18}.

Thanks to \eqref{eq:base Gr}, \eqref{eq:Hesseta}, \eqref{eq:gradeta}  and \eqref{eq:Hessalpha}, we deduce the following pointwise inequalities in $G_r$ for any $i,j,k=1,\ldots, N$:
\begin{equation*}
	|f_{ij}^k| \le C(N) |\Hess u| \le C(N) \delta\, ;
\end{equation*}
\begin{equation*}
	\abs{\alpha_{ij}^k + \Hess u_k^x(\nabla u^x_{i}, \nabla u^x_{j})}
	\le C(N) \eta \sum_{i',j'} |f_{i',j'}^k| 
	\le C(N)\eta \delta \, .
\end{equation*}
Using again \eqref{eq:gradeta}  and \eqref{eq:Hessalpha}, we deduce 
\begin{equation*}
	|\alpha_{ij}^k + f_{ij}^k(y)| \le C(N)\eta \delta \, ,
	\qquad
	\text{for any $y\in G_r$} \, ,
\end{equation*}
which gives in turn
\begin{equation}\label{eq:p1}
	\abs{\Hess u_k^x + \alpha_{ij}^k \nabla u_i^x \otimes \nabla u_j^x} \le C(N)\eta \delta 
	\qquad
	\text{in $G_r$} \, .
\end{equation}
We finally observe that
\begin{equation}\label{eq:p2}
	\begin{split}
	\int_{B_r(x)\setminus G_r} & |\Hess u|^2 \di \haus^N
	\\& \le \int_{B_r(x)\setminus \tilde E} |\Hess u|^2 \di \haus^N
	+ \int_{(B_r(x)\cap \tilde E) \setminus G_r} |\Hess u|^2 \di \haus^N
	\\&
	\le \eta\haus^N(B_r(x)) + C(N) \delta \haus^N(B_r(x)\setminus G_r)
	\\&
	\le C(N)\eta \haus^N(B_r(x)) \, ,
	\end{split}
\end{equation}
where we used \eqref{eq:Greta}, the fact that $|\Hess u|^2 \le\delta$ in $\tilde E$, and \eqref{eq:Hesseta}.

By combining \eqref{eq:p1} and \eqref{eq:p2}, we obtain
\begin{equation}
	\fint_{B_r(x)} \abs{\Hess u_k^x + \alpha_{ij}^k \nabla u_i^x \otimes \nabla u_j^x}^2 \di \haus^N  \le C(N)\eta \, ,
\end{equation}
which implies the sought conclusion due to the arbitrariness of $\eta$ and $r\le r_0(\eta)$.
\end{proof}

Given any point $x\in E$ as in the statement of \autoref{prop:goodHessian}, up to the addition of a constant that does not affect the forthcoming statements we can assume that $u^x(x)=0\in\setR^N$.
We introduce the function $v:B_1(p)\to \setR^N$ by setting
\begin{equation}\label{eq:defv}
v_k(y):=u^x_k(y)+\frac{1}{2}\sum_{ij}\alpha^k_{ij} u^x_i(y) u^x_j(y)\, ,
\qquad
\text{for all $k=1,\dots, N$} \, .
\end{equation}
The point $x\in E$ as in the statement of \autoref{prop:goodHessian} will be fixed from now on, so there will be no risk of confusion.

\smallskip

Below we are concerned with the properties of the function $v$ as in \eqref{eq:defv}. Notice that, on a smooth Riemannian manifold, $v$ would have vanishing Hessian at $x$ and verify $\nabla v_i(x)\cdot\nabla v_j(x)=\delta_{ij}$, by its very construction. 

\begin{lemma}\label{lemma:algebraicomp}
Under the same assumptions and with the same notation introduced above, the map $v:B_1(p)\to\setR^N$  as in \eqref{eq:defv} has the following properties:
\begin{itemize}
\item[i)] for any $i,j=1,\dots,N$ it holds
\begin{equation}\label{eq:normdelta}
\nabla v_i(x)\cdot\nabla v_j(x)=\delta_{ij}\, ,\quad \lim_{t\downarrow 0}\fint_{B_t(x)}\abs{\nabla v_i\cdot\nabla v_j-\delta_{ij}}\di\haus^N=0\, ;
\end{equation}
\item[ii)] for any $k=1,\dots, N$, it holds 
\begin{equation}\label{eq:hess0}
\lim_{t\downarrow 0}\fint_{B_t(x)}\abs{\Hess v_k}^2\di\haus^N=0\, .
\end{equation}
\end{itemize}
\end{lemma}

\begin{proof}
Employing the standard calculus rules, let us compute the derivatives of $v$:
\begin{equation}\label{eq:gradv}
\nabla v_k=\nabla u^x_k+\frac{1}{2}\sum_{i,j}\alpha^k_{ij}\left(u_i^x\nabla u^x_j+u^x_j\nabla u^x_i\right)
\end{equation}
\begin{equation}\label{eq:Hessv}
\Hess v_k=\Hess u^x_k+\sum_{i,j}\alpha^k_{ij}\left(u^x_i\Hess u^x_j+\nabla u^x_i\otimes\nabla u_j^x\right)\, .
\end{equation}
As $u^x(x)=0$, $\nabla u^x_i(x)\cdot\nabla u^x_j(x)=\delta_{ij}$ and \eqref{eq:Lebnablauiuj} holds by construction, \eqref{eq:gradv} shows that 
\begin{equation*}
\nabla v_i(x)\cdot \nabla v_j(x)=\delta_{ij}\, 
\end{equation*}
and 
\begin{equation*}
\lim_{t\downarrow 0}\fint_{B_t(x)}\abs{\nabla v_i\cdot\nabla v_j-\delta_{ij}}  \di\haus^N=0\, .
\end{equation*}
Then we estimate 
\begin{equation}\label{eq:esthessvk2}
\abs{\Hess v_k}^2\le 2\abs{\Hess u^x_k+\sum_{i,j}\alpha^k_{ij}\nabla u^x_i\otimes\nabla u^x_j}^2+2\abs{\sum_{i,j}\alpha^k_{ij}u_i^x\Hess u^x_j}^2\,.
\end{equation}
Integrating \eqref{eq:esthessvk2} over $B_t(x)$ and using the uniform Lipschitz estimates for $u$ 
\begin{equation*}
\abs{u^x_i(y)-u^x_i(x)}\le C(N)\dist(x,y)\, ,\quad\text{for any $i=1,\dots, N$},
\end{equation*}
we obtain
\begin{align*}
\fint_{B_t(x)}\abs{\Hess v_k}^2\di\haus^N\le &\; 2\fint_{B_t(x)}\abs{\Hess u^x_k+\sum_{i,j}\alpha^k_{ij}\nabla u^x_i\otimes\nabla u^x_j}^2\di\haus^N\\
&\; +C(N) t^2\fint_{B_t(x)}\abs{\Hess u^x}^2\di\haus^N\, .
\end{align*} 
By using  \eqref{eq:vansihHess} and \autoref{prop:goodHessian} (ii)(b), we conclude that
\begin{equation*}
\lim_{t\downarrow 0}\fint_{B_t(x)}\abs{\Hess v_k}^2\di\haus^N=0\, ,\quad\text{for any $k=1,\dots, N$}\, .
\end{equation*}

\end{proof}

\subsection{Volume estimates via almost splitting map with vanishing Hessian}
\label{subsec:volume via almost splitting}
Given  $v:B_1(p)\to\setR^N$ as in \autoref{lemma:algebraicomp} we introduce
 the function $r:B_1(p)\to [0,\infty)$ by 
\begin{equation}\label{eq:introrv}
r^2(y):=\sum_iv^2_i(y)\, .
\end{equation}
We aim at showing that $r$ is a polynomially good approximation of the distance from $x$ and, at the same time, it is an approximate solution of $\Delta r^2=2N$, in integral sense. With some algebraic manipulations and the standard chain rules, we obtain the following.

\begin{lemma}
With the same notation as above the following hold:
\begin{itemize}
\item[i)]
\begin{equation}\label{eq:laplar2}
\Delta r^2=2\sum_i\left[\abs{\nabla v_i}^2+v_i\Delta v_i\right]\, ;
\end{equation}
\item[ii)]
\begin{equation}\label{eq:nablar}
\abs{\abs{\nabla r}^2-1}\le \sum_{i,j}\abs{\nabla v_i\cdot\nabla v_j-\delta_{ij}}\, .
\end{equation}
\end{itemize}
\end{lemma}

\begin{proof}
The expression for the Laplacian \eqref{eq:laplar2} follows from the chain rule by the very definition $r^2=\sum_iv_i^2$.
\smallskip

In order to obtain the gradient estimate, we compute 
\begin{equation*}
\nabla \sum_iv_i^2=2\sum v_i\nabla v_i\, .
\end{equation*}
Hence
\begin{equation*}
\abs{\nabla r^2}^2=4\sum_{i,j}v_iv_j\nabla v_i\cdot\nabla v_j\, . 
\end{equation*}
Then we split
\begin{equation*}
\abs{\nabla r^2}^2=4\sum_{i,j}v_iv_j\delta_{ij}+4\sum_{i,j}v_iv_j\left(\nabla v_i\cdot\nabla v_j-\delta_{ij}\right)\, .
\end{equation*}
Hence 
\begin{equation*}
\abs{\abs{\nabla r^2}^2-4r^2}=4\abs{\sum_{i,j}v_iv_j\left(\nabla v_i\cdot\nabla v_j-\delta_{ij}\right)}\le 4r^2\sum_{i,j}\abs{\nabla v_i\cdot\nabla v_j-\delta_{ij}}\, , 
\end{equation*}
eventually proving \eqref{eq:nablar}.
\end{proof}

We will rely on the following technical result.

\begin{lemma}\label{lemma:doubling}
Let $(X,\dist,\haus^N)$ be an $\RCD(-(N-1),N)$ metric measure space. Let $v>0$. There exists a constant $C(N,v)>0$ such that for any $x\in X$, if $\haus^{N}(B_{3/2}(x)\setminus B_1(x))>v$, then for any nonnegative function $f\in L^{\infty}(B_1(x))$ and for almost every $0<t<1$ it holds
\begin{equation}\label{eq:estintfgeod}
\fint_{\partial B_t(x)}\int_0^tf(\gamma_y(s))\di s\di\haus^{N-1}(y)\le C(N,v)t\sup_{0<s<t}\fint_{B_s(x)}f\di\haus^N\, ,
\end{equation}
where for $\haus^N$-a.e. $y\in B_1(x)$ we denote by $\gamma_y$ the unique minimizing geodesic between $\gamma_y(0)=x$ and $\gamma_y(\dist(x,y))=y$. 
\end{lemma}

\begin{proof}
Under the assumption that $\haus^{N}(B_{3/2}(x)\setminus B_1(x))>v$, the following inequalities hold, by Bishop-Gromov monotonicity (for spheres) and the coarea formula:
\begin{equation}\label{eq:sph}
C(N,v)\le \frac{\haus^{N-1}(\partial B_t(x))}{t^{N-1}}\le C(N)\, ,\quad \text{for a.e. $0<t<1$}\, ,
\end{equation}
\begin{equation}\label{eq:vol}
C(N,v)\le \frac{\haus^N(B_t(x))}{t^N}\le C(N)\, ,\quad \text{for any $0<t<1$}
\end{equation}
and 
\begin{equation}\label{eq:crown}
C(N,v)\le \frac{\haus^N(B_{(1+\eps)t}(x)\setminus B_{t}(x))}{\eps t^N}\le C(N)\, ,
\end{equation}
for any $0<t<1$ and any $0<\eps<1/10$.

	It is enough to check that for any  nonnegative continuous function $f$ on $B_1(x)$ it holds
	\begin{equation}\label{z9}
		\fint_{\partial B_t(x)} f(\gamma_y(s)) \di \haus^{N-1}(y) \le C(N,v) \fint_{\partial B_s(x)} f(y)\di \haus^{N-1}(y)
	\end{equation}
     for a.e.  $0< s \le t\leq 1$. Indeed, if \eqref{z9} holds, then 
    \begin{equation*}
    	\begin{split}
    		\fint_{\partial B_t(x)}\int_0^tf(\gamma_y(s))\di s\di\haus^{N-1}(y) &
    		= \int_0^t \left( \fint_{\partial B_t(x)} f(\gamma_y(s)) \di \haus^{N-1}(y) \right) \di t
    		\\&
    		\le C(N,v) \int_0^t \fint_{\partial B_s(x)} f\di \haus^{N-1} \di s
    		\\&
    		\le C(N,v)t \sup_{0<s<t}\fint_{B_s(x)}f\di\haus^N \, ,
    	\end{split}
    \end{equation*}
 where we used \eqref{eq:sph} and \eqref{eq:vol} for the last inequality.\\
  Let us show \eqref{z9}. From the ${\mathsf{MCP}}(-(N-1), N)$ property (which is satisfied by $\CD(-(N-1),N)$ spaces and a fortiori for $\RCD(-(N-1),N)$ spaces) and \eqref{eq:crown}, we have
      \begin{equation*}%\label{z9Vol}
		\fint_{B_{(1+\eps)t}(x)\setminus B_{t}(x) } f(\gamma_y (s\, \dist(x,y)/t)) \di \haus^{N}(y) \le C(N,v) \fint_{B_{(1+\eps)s}(x)\setminus B_{s}(x) } f(y)\di \haus^{N}(y)
	\end{equation*}
    for any $0< s \le t\leq 1$. Passing to the limit as $\eps\downarrow 0$, taking into account the classical weak convergence of the normalized volume measure of the tubular neighbourhood to the surface measure for spheres, we obtain that \eqref{z9} holds for a.e. $0< s \le t \le 1$.
\end{proof}

\begin{proposition}\label{prop:asymptoticestimates}
Under the same assumptions and with the same notation as above, the following asymptotic estimates hold for the function $r:B_1(x)\to[0,\infty)$:
\begin{itemize}
\item[i)]
\begin{equation}\label{eq:rmeano}
\fint_{\partial B_t(x)}r\di\haus^{N-1}\le t+o(t^2)\,,  \quad\text{as $t\downarrow 0$}\, ;
\end{equation}
\item[ii)]
\begin{equation}\label{eq:nablarmeano}
\fint_{B_t(x)}\abs{\abs{\nabla r}-1}\di\haus^N=o(t)\, , \quad\text{as $t\downarrow 0$}\, ;
\end{equation}
\item[iii)] 
\begin{equation}\label{eq:laplar2meano}
\fint_{B_t(x)}\abs{\Delta r^2-2N}\di\haus^N=o(t)\, ,\quad\text{as $t\downarrow 0$}\, .
\end{equation}
\end{itemize}
\end{proposition}

\begin{proof}
We start noticing that a standard application of the Poincar\'e inequality from \cite{VonRenesse08,Rajala12}, in combination with \eqref{eq:normdelta} and \eqref{eq:hess0}, shows that for any $i,j=1,\dots,N$ it holds 
\begin{equation}\label{eq:normquick}
\fint_{B_t(x)}\abs{\nabla v_i\cdot\nabla v_j-\delta_{ij}}\di\haus^N=o(t)\, ,\quad\text{as $t\downarrow 0$}\, .
\end{equation}
Given \eqref{eq:normquick}, \eqref{eq:nablarmeano} follows from \eqref{eq:nablar}.

In order to prove \eqref{eq:laplar2meano}, we employ \eqref{eq:laplar2} to estimate
\begin{equation*}
\abs{\Delta r^2-2N}\le 2 \sum_{i}\abs{1-\abs{\nabla v_i}^2}+2\sum_i\abs{v_i}\abs{\Delta v_i}\, .
\end{equation*}
Hence
\begin{equation*}
\fint_{B_t(x)}\abs{\Delta r^2-2N}\di\haus^N\le  2 \sum_i\fint_{B_t(x)}\abs{1-\abs{\nabla v_i}^2}\di\haus^N+2\sum_i\fint_{B_t(x)}\abs{v_i}\abs{\Delta v_i}\di\haus^N\, .
\end{equation*}
The first summand above can be dealt with via \eqref{eq:normquick}. In order to bound the second one we notice that 
\begin{align*}
\fint_{B_t(x)}\abs{v_i}\abs{\Delta v_i}\di\haus^N\le &C(N)t\fint_{B_t(x)}\abs{\Delta v_i}\di\haus^N\\
\le &C(N) t\fint_{B_t(x)}\abs{\Hess v_i}\di\haus^N\\
\le &C(N)t\left(\fint_{B_t(x)}\abs{\Hess v_i}^2\di\haus^N\right)^{\frac{1}{2}}\, ,
\end{align*}
where we used that $v_i$ is $C(N)$-Lipschitz with $v_i(x)=0$ and we rely on the known identity
\begin{equation*}
\Delta v=\tr\Hess v\, ,\quad\text{$\haus^N$-a.e.}
\end{equation*}
for $\RCD(K,N)$ spaces $(X,\dist,\haus^N)$, see \cite{Han18,DePhilippisGigli18}.
All in all this proves \eqref{eq:laplar2meano}.

\medskip

We are left to prove \eqref{eq:rmeano}.
In order to do so, we consider geodesics $\gamma_y$ from $x$ to $y$, where $\gamma(0)=x$ and $\gamma(t)=y$. This geodesic is unique for $\haus^N$-a.e. $y$, hence it is unique for $\haus^{N-1}$-a.e. $y\in\partial B_t(x)$ for $\Leb^1$-a.e. $t>0$. Then we notice that for $\haus^N$-a.e. $y\in B_1(x)$ (hence for $\haus^{N-1}$-a.e. $y\in \partial B_t(x)$ for $\Leb^1$-a.e. $t>0$)  it holds 
\begin{equation*}
\abs{r(y)-r(x)}\le \int_0^t\abs{\frac{\di}{\di s}r\circ\gamma_y(s)}\di s\le \int_0^t\abs{\nabla r}(\gamma_y(s))\di s\, . 
\end{equation*}
Now we can integrate on $\partial B_t(x)$ and get
\begin{align}
\nonumber \int_{\partial B_t(x)}&r(y)\di\haus^{N-1}(y)
\\\nonumber  &=\int_{\partial B_t(x)}\abs{r(y)-r(x)}\di\haus^{N-1}(y)\\
\nonumber &\le \int_{\partial B_t(x)}\int_0^t\abs{\nabla r}(\gamma_y(s))\di s\di\haus^{N-1}(y)\\
&\le  t\haus^{N-1}(\partial B_t(x))
+\int_{\partial B_t(x)}\int_0^t \abs{\abs{\nabla r}-1}(\gamma_y(s))\di s\di\haus^{N-1}(y)\, .\label{eq:boundavr}
\end{align}
In order to deal with the last summand, we notice that, by \autoref{lemma:doubling},
\begin{equation}\label{eq:tomean}
\fint_{\partial B_t(x)}\int_0^t \abs{\abs{\nabla r}-1}(\gamma_y(s))\di s\di\haus^{N-1}(y)\le C(N)t \sup_{0<s<t}\fint_{B_s(x)}\abs{\abs{\nabla r}-1}\di\haus^N\, . 
\end{equation}
The combination of \eqref{eq:boundavr} and \eqref{eq:tomean} proves that
\begin{equation}\label{eq:avr2}
\fint_{\partial B_t(x)}r\di\haus^{N-1}\le t+C(N)t \sup_{0<s<t} \fint_{B_s(x)}\abs{\abs{\nabla r}-1}\di\haus^N\, .
\end{equation}
Notice that the lower volume bound $\haus^N(B_{3/2}(x)\setminus B_1(x))>C(N)$ is satisfied by volume convergence, under the current assumptions.
Taking into account \eqref{eq:nablarmeano}, \eqref{eq:avr2} shows that \eqref{eq:rmeano} holds.
\end{proof}

\subsection{Proof of \autoref{prop:or}}

The proof is divided into three steps. In the first step, we show that on each $\delta$-regular ball at least half of the points satisfy
\begin{equation}\label{z12}
	\lim_{r\downarrow 0}\frac{1}{r}\abs{1-\frac{\haus^N(B_r(x))}{\omega_Nr^N}}=0
	\, .
\end{equation}
Then, in Step 2, we bootstrap this conclusion to get \eqref{z12} at $\haus^N$-a.e. $x$ in a $\delta$-regular ball. We conclude in Step 3 by covering $X$ with $\delta$-regular balls up to a $\haus^N$-negligible set.
\medskip

\textbf{Step 1.} Let $\delta = \delta(N,1/5)$ as in \autoref{prop:goodHessian}.
We claim that for any $\delta$-regular ball  $B_{4r}(p)\subset X$ there exists $E\subset B_r(p)$ such that
\begin{equation}\label{z11}
\haus^N(B_r(p)\setminus E) \le \frac{1}{5} \haus^N(B_r(p))\, , 
\end{equation}
\begin{equation}
	\lim_{r\downarrow 0}\frac{1}{r}\abs{1-\frac{\haus^N(B_r(x))}{\omega_Nr^N}}=0\, ,
	\qquad
	\text{for any $x\in E$} \, .
\end{equation}

\smallskip

Indeed, we can apply \autoref{prop:goodHessian} and find $E\subset B_r(p)$ satisfying \eqref{z11}. Moreover, for any $x\in E$ there exists a function $v:B_r(x)\to\setR^N$ that, up to scaling, has all the good properties guaranteed by \autoref{lemma:algebraicomp}.

Under these assumptions, we can apply \eqref{eq:integratedgeneral} to the map $r$ introduced in \eqref{eq:introrv} in terms of $v$. Recalling \autoref{prop:asymptoticestimates} (see also \eqref{eq:lower} and \eqref{eq:lower2} for the estimate for the negative part), \eqref{eq:integratedgeneral} shows that 
\begin{equation}
\lim_{r\downarrow 0}\frac{1}{r}\abs{1-\frac{\haus^N(B_r(x))}{\omega_Nr^N}}=0\, .
\end{equation}
\medskip

\textbf{Step 2.} Let $\delta'\le \delta'(\delta, N)$ with the property that if $B_{10}(p)$ is $\delta'$-regular then $B_r(x)$ is $\delta$-regular for any $x\in B_2(p)$ and $r<5$. We prove that if $B_{10}(p)$ is a $\delta'$-regular ball, then 
	\begin{equation}
		A_\eta:= \left\lbrace x\in B_1(p)\, : \,  \limsup_{r\downarrow 0} \frac{1}{r}\abs{1- \frac{\haus^N(B_r(x))}{\omega_N r^N}} > \eta\right\rbrace
	\end{equation}
	is $\haus^N$-negligible for any $\eta>0$.

\smallskip

Let us argue by contradiction. If this is not the case, we can find a Lebesgue point $x\in A_\eta$. In particular, there exists $r<1$ such that $\haus^N(B_r(x)\cap A_\eta)>\frac{1}{5} \haus^N(B_r(x))$. As $B_r(x)$ is a $\delta$-regular ball, the latter inequality contradicts Step 1.

\medskip

\textbf{Step 3.} We conclude the proof by observing that there is a family of $\delta'$-regular balls $\{B_{10 r_i}(x_i)\}_{i\in \setN}$ such that $\haus^N(X\setminus \bigcup_i B_{r_i}(x_i))=0$. This conclusion follows by a Vitali covering argument after recalling that $\haus^N$-a.e. point in $X$ has Euclidean tangent cone and taking into account the classical $\eps$-regularity theorem for almost Euclidean balls, see \cite{CheegerColding97,DePhilippisGigli18}.

\section{Control of the metric measure boundary}\label{sec:control of the mmb}

This section is devoted to the proof of the $\eps$-regularity \autoref{prop:boundary on regular balls}. The proof is divided into two main parts: an iteration lemma, where we establish uniform bounds for the deviation measures and vanishing of the metric measure boundary on an almost regular ball away from a set of small $(N-1)$-dimensional content; the iterative application of the lemma to establish the bounds and the limiting behaviour on the full ball.

\begin{lemma}[Iteration Lemma]\label{lemma:iteration}
	\label{lemma:Iteration}
	For every $\eps>0$, if $\delta \le \delta(N,\eps)$ the following property holds.
	If $(X,\dist,\haus^N)$ is an $\RCD(-\delta (N-1),N)$ m.m.s. and $B_{20}(p) \subset X$ is a $\delta$-regular ball, then there exists a Borel set $E\subset B_2(p)$ with the following properties:
	\begin{itemize}
		\item[i)] 	$|\mu_r| (E) \le \eps$
		for any $0<r<1$;
		\item[ii)] $B_1(p)\setminus E\subset \bigcup_i B_{r_i}(x_i)$ and $\sum_ir_i^{N-1}\le \eps$;
		\item[iii)] $\abs{\mu_r}(E)\to 0$ as $r\downarrow 0$.
	\end{itemize}
\end{lemma}

We postpone the proof of the iteration lemma and see how to establish the $\eps$-regularity theorem by taking it for granted.

\subsection{Proof of \autoref{prop:boundary on regular balls} given \autoref{lemma:iteration}}
	Let us fix $r\in (0,1)$, and $\eps\le 1/5$, $\delta \le \delta(N,\eps)$ as in \autoref{lemma:Iteration}.
	 We assume that $B_{10}(p)$ is $\delta'$-regular for some $\delta'=\delta'(\delta,N)>0$ small enough so that $B_s(x)$ is $\delta$-regular for any $x\in B_5(p)$ and $s\le 5$. 
	 
	 We apply \autoref{lemma:Iteration} to get a Borel set $E_1\subset B_3(p)$ such that
	 \begin{itemize}
	 	\item[(a)] $|\mu_r|(E_1) \le \eps$;
	 	\item[(b)] $B_2(p)\setminus E_1 \subset \bigcup_a B_{r_a}(x_a) \cup \bigcup_b B_{r_b}(x_b) $ with $\sum_a r_a^{N-1} + \sum_b r_b^{N-1} \le \eps$;
	 	\item[(c)] $r_a \le r$ and $r_b > r$.
	 \end{itemize}
	 We set 
	 \begin{equation}\label{eq:defG1}
	 	G_1 := E_1 \cup \bigcup_a B_{r_a}(x_a) \, ,
	 \end{equation}
	 and observe that
	 \begin{equation*}
	 	\begin{split}
	 	|\mu_r|(G_1) & \le |\mu_r|(E_1) + \sum_a |\mu_r|(B_{r_a}(x_a))
	 	\\& \le \eps + \sum_a \frac{1}{r}\int_{B_{r_a(x_a)}} \abs{1-\frac{\haus^N(B_r(x))}{\omega_N r^N}}\di \haus^N(x)
	 	\\& \le \eps + C(N) \sum_a \frac{\haus^N(B_{r_a}(x_a))}{r}
	 	\\& \le \eps + C(N) \sum_a r_a^{N-1} \, .
	 	\end{split}
	 \end{equation*}
	  By (b) we deduce the existence of a constant $c(N)\ge 1$ such that
	  \begin{equation}\label{eq:estG1}
	  	|\mu_r|(G_1) \le c(N) \eps \, .
	  \end{equation}
      To control $|\mu_r|(B_{r_b}(x_b))$, we apply again \autoref{lemma:Iteration} to any ball $B_{r_b}(x_b)$. Arguing as above, we obtain a set $G_{b}$ (constructed analogously to $G_1$ as in \eqref{eq:defG1}) such that
      \begin{itemize}
      	\item[(a')] $|\mu_r|(G_{b})\le c(N)\eps r_b^{N-1}$;
      	\item[(b')] $B_{r_b}(x_b)\setminus G_{b} \subset \bigcup_{b_1} B_{r_{b,b_1}}(x_{b,b_1})$ and $\sum_{b_1} r_{b,b_1}^{N-1} \le \eps r_b^{N-1}$;
      	\item[(c')] $r_{b,b_1} > r$.
      \end{itemize}
       After two steps of the iteration we are left with a good set 
       \begin{equation*}
       	G_2 := G_1 \cup \bigcup_b G_b \, ,
       \end{equation*}
       such that
       \begin{equation*}
       	|\mu_r|(G_2) \le c(N) \eps + \sum_b |\mu_r|(G_b)
       	\le c(N)(\eps + \eps^2) \, ,
       \end{equation*}
       as a consequence of \eqref{eq:estG1}, (a') and (b). Moreover, 
       \begin{equation*}
       	B_2(p)\setminus G_2 \subset \bigcup_b \bigcup_{b_1} B_{r_{b,b_1}}(x_{b,b_1})\, ,
       \end{equation*}
       \begin{equation*}
       	\sum_{b,b_1} r_{b,b_1}^{N-1} 
       	\le \eps \sum_{b} r_b^{N-1} 
       	\le \eps^2 \, . 
       \end{equation*}
       If the family of bad balls $B_{r_{b,b_1}}(x_{b,b_1})$ is not empty, we iterate this procedure.
       At the $k$-th step, we have a good set $G_k$ such that
       \begin{equation*}
       	|\mu_r|(G_k)\le c(N)(\eps + \eps^2 + \ldots + \eps^k) \, ,
       \end{equation*}
       and bad balls satisfying
       \begin{equation*}
       	B_2(p)\setminus G_k \subset \bigcup_i B_{r_{i,k}}(x_{i,k}) \, ,
       	\quad \sum_k r_{i,k}^{N-1} \le \eps ^k \, , \quad r_{i,k} > r \, \, \, \forall \, i\in \setN \, .
       \end{equation*}
       Notice that $r_{i,k} \le \eps^{\frac{k}{N-1}}$, hence this procedure must stop after $M$ steps, for some $M\le (N-1) \frac{\log r}{\log \eps}$. Therefore, $B_2(p)\subset G_M$ and
       \begin{equation}\label{z4}
       	|\mu_r|(B_2(p)) 
       	\le c(N)(\eps +  \ldots + \eps^M)
       	\le 2c(N)\eps \, .
       \end{equation}
     The proof of \eqref{eq:uniformboundsep} is completed.
     
     \medskip

     Let us now prove that $|\mu_r|(B_1(p))\to 0$ as $r\downarrow 0$.	
     As a consequence of \eqref{z4}, we can extract a weak limit in $B_2(p)$
     \begin{equation*}
     	|\mu_{r_i}| \to \mu \, , \qquad
     	\text{as $r_i\to 0 $} \, .
     \end{equation*}
     By the scale invariant version of \eqref{eq:uniformboundsep}, we deduce
     \begin{equation}\label{z5}
        \mu(B_s(x)) \le \eps s^{N-1} \, ,
        \qquad
        \text{for any $x\in B_1(p)$ and $s<1$} \, .
     \end{equation}
     To conclude the proof, it is enough to show that $\mu(B_{1}(p))=0$.
     To this aim we apply an iterative argument analogous to the one above. Using the iteration \autoref{lemma:iteration}, we cover
     \begin{equation*}
     	B_1(p) \setminus E \subset \bigcup_i B_{r_i}(x_i) \, , 
     	\qquad
     	\sum_i r_i^{N-1}\le \eps \, ,
     \end{equation*}
     and observe that
     \begin{equation*}
     	\mu(B_1(p))\le \mu(E) + 
     	\sum_i \mu(B_{r_i}(x_i)) \le \eps\sum_i r_i^{N-1} \le \eps^2\, ,
     \end{equation*}
     where we used \eqref{z5} and \autoref{lemma:iteration} (iii).
     
     We apply the same decomposition to each ball
     \begin{equation*}
     	B_{r_i}(x_i)\setminus E_i \subset \bigcup_j B_{r_{i,j}}(x_{i,j}) \, ,
     	\qquad
     	\sum_j r_{i,j}^{N-1} \le \eps r_i^{N-1} \, ,
     \end{equation*}
     obtaining
     \begin{equation*}
     	\mu(B_1(p)) = \sum_{i,j} \mu(B_{r_{i,j}}(x_{i,j})) 
     	\le \eps \sum_{i,j} r_{i,j}^{N-1}
     	\le \eps^3 \, .
     \end{equation*}
     After $k$ steps of the iteration we deduce $\mu(B_1(p)) \le \eps^{k+1}$. We conclude by letting $k\to \infty$.

\subsection{Proof of \autoref{lemma:Iteration}}

Let $\delta'=\delta'(\eps,N)>0$ to be chosen later.
If $\delta \le \delta(\delta',\eps,N)$, we can build a $\delta'$-splitting map $u:B_{10}(p)\to\setR^N$. We consider the set $E$ of those points $x\in B_1(p)$ such that 
\begin{equation*}
\sup_{0<r<5}r\fint_{B_r(x)}\abs{\Hess u}^2\di\haus^N\le \eta\, ,
\end{equation*}
where $\eta=\eta(N)$ will be chosen later.
For any $x\in E$ and $r\le 1$, the Poincar\'e inequality gives
\begin{align*}
	\abs{ \fint_{B_{2r}(x)} \mathcal{E}\di \haus^N - \fint_{B_r(x)} \mathcal{E}\di \haus^N }
	& \le C(N) r \fint_{B_{3r}(x)} |\Hess u|\di \haus^N
	\\& 
	\le C(N)\left(  r^2 \fint_{B_{3r}(x)} |\Hess u|^2\di \haus^N  \right)^{1/2}
	\\& 
	\le C(N) \eta^{1/2} \, ,
\end{align*}
which along with a telescopic argument (cf. with the proof of \cite[Lemma 4.16]{BrueNaberSemola22}) gives
\begin{equation*}
	\fint_{B_r(x)} \mathcal{E} \di \haus^N \le \delta' + C(N)\eta^{1/2} \, ,
	\quad\text{for any $x\in E$ and $r<5$} \, .
\end{equation*}
We assume $\eta=\eta(N)$ and $\delta'=\delta'(\eps,N)$ small enough so that $\delta' + C(N)\eta^{1/2}\le \delta_0(N)$, where the latter is given by \autoref{lemma:linear algebra}. 
For any $x\in E$, we apply \autoref{lemma:linear algebra} with $A_{ij}= \nabla u_i(x)\cdot \nabla u_j(x)$ and we get a matrix $B_x\in \setR^{N\times N}$ such that $|B_x|\le C(N)$ and $v:= B_x \circ u: B_{10}(p)\to\setR^N$ verifies
\begin{equation*}
	\nabla v_i(x)\cdot \nabla v_j(x)=\delta_{ij}\, ,\quad i,j=1,\dots, N\, .
\end{equation*}
Since, by construction, any point $x\in E$ is a Lebesgue point for $\nabla u_i \cdot \nabla u_j$ (and thus for $\nabla v_i \cdot \nabla v_j$), it follows that
\begin{equation*}
\lim_{s\downarrow 0} \fint_{B_s(x)}\abs{\nabla v_i\cdot\nabla v_j-\delta_{ij}}\di\haus^N=0\,.
\end{equation*}
Moreover, it is clear that $|\Hess v|\leq C(N) |\Hess u|$. 

Applying again a telescopic argument based on the Poincar\'e inequality we infer that 
\begin{align}
\fint_{B_r(x)}&\abs{\nabla v_i\cdot\nabla v_j-\delta_{ij}}\di\haus^N \nonumber
\\& \le \lim_{s\downarrow 0} \fint_{B_s(x)}\abs{\nabla v_i\cdot\nabla v_j-\delta_{ij}}\di\haus^N
+ C(N)r \sup_{s<10}\fint_{B_s(x)}\abs{\Hess v}\di \haus^N \nonumber 
\\&
\le C(N)r \sup_{s<10}\fint_{B_s(x)}\abs{\Hess v}\di \haus^N \nonumber
\\ &  \le C(N)r \sup_{s<10}\fint_{B_s(x)}\abs{\Hess u}\di \haus^N  \nonumber
\\& = C(N)r M_{10} |\Hess u|(x) 
\, , \label{eq:nablavivjM10Hes}
\end{align}
for any $0<r<1$, where 
\begin{equation}
M_{10}|\Hess u|(x):= \sup_{s<10}\fint_{B_s(x)}\abs{\Hess u}\di \haus^N
\end{equation}
is the maximal function of $\abs{\Hess u}$.

Combining  \autoref{cor:gapBGharm} with \eqref{eq:nablavivjM10Hes} gives
\begin{equation*}
\frac{1}{r}\left(1-\frac{\haus^N(B_r(x))}{\omega_Nr^N}\right)\le C(N)M_{10}\abs{\Hess u}(x)\, ,
\end{equation*}
for $\haus^N$-a.e. $x\in E$ and for any $0<r<1$. 
In particular, 
\begin{equation*}
\mu_r\res E\le C(N)M_{10}\abs{\Hess u}\haus^N\res E\, ,
\qquad
\text{for any $0<r<1$} \, .
\end{equation*}
The classical $L^2$ maximal function estimate gives
\begin{equation}\label{eq:contronorm}
\mu_r(E)\le C(N)\int_{B_1(p)}M_{10}\abs{\Hess u}\di\haus^N\le C(N)\left(\fint_{B_{20}(p)}\abs{\Hess u}^2\di\haus^N\right)^{\frac{1}{2}}\, .
\end{equation}
Therefore, $M_{10}\abs{\Hess u}$ is integrable and dominates uniformly the sequence
\begin{equation*}
f_r:=\frac{1}{r}\abs{1-\frac{\haus^N(B_r(x))}{\omega_Nr^N}} \, ,
\quad
r\in (0,1)\, ,
\end{equation*}
on $E$. Moreover, $f_r(x) \to 0$ as $r\downarrow 0$
for $\haus^N$-a.e. $x\in X$ by \autoref{prop:or}. Hence by the dominated convergence theorem 
\begin{equation*}
\abs{\mu_r}(E) \to 0 \, ,
\qquad
\text{as $r\downarrow 0$} \, ,
\end{equation*}
proving (iii).

\medskip

In order to get the $(N-1)$-dimensional content bound (ii), we employ a weighted maximal function argument (see for instance the proof of \cite[Proposition 4.19]{BrueNaberSemola22}) to show that $B_1(p)\setminus E$ can be covered by a countable union of balls $\bigcup_iB_{r_i}(x_i)$ with 
\begin{equation*}
\sum_ir_i^{N-1}\le \delta''(\delta',N)\, .
\end{equation*}
Indeed, for any $x\in B_1(p)$ such that 
\begin{equation*}
\sup_{0<r<1}r\fint_{B_r(x)}\abs{\Hess u}^2\di\haus^N>\eta(N)\, ,
\end{equation*}
we set $r_x>0$ to be the maximal radius such that 
\begin{equation*}
\frac{r_x}{5}\fint _{B_{r_x/5}(x)}\abs{\Hess u}^2\di\haus^N\ge \eta(N)\, .
\end{equation*}
By Ahlfors regularity of $\haus^N$, we immediately deduce that
\begin{equation}
\int_{B_{r_x/5}(x)}\abs{\Hess u}^2\di\haus^N\ge C(N) r_x^{N-1}\, .
\end{equation}
By a Vitali covering argument, we can cover $E$ with a countable union $B_{r_i}(x_i)$ such that $B_{r_i/5}(x_i)$ are disjoint. Then
\begin{align}
\nonumber \sum_ir_i^{N-1}\le & C(N)\sum_i \int_{B_{r_i/5}(x_i)}\abs{\Hess u}^2\di\haus^N\\
\le & C(N)\int_{B_2(p)}\abs{\Hess u}^2\di\haus^N\le C(N)\delta'\, .\label{eq:controbad}
\end{align}
This completes the proof of (i) and (ii) after choosing $\delta'=\delta'(\eps,N)$ small enough so that the right hand sides in \eqref{eq:contronorm} and \eqref{eq:controbad} are smaller than $\eps$.

\section{Spaces with boundary}\label{sec:spaces with boundary}

In this section we aim at controlling the metric measure boundary on $\RCD(-(N-1),N)$ spaces $(X,\dist,\haus^N)$ with boundary satisfying fairly natural regularity assumptions.

Let $0<\delta\le 1$ be fixed.
We consider an $\RCD(-\delta(N-1),N)$ space $(X,\dist,\haus^N)$ with boundary and we recall that a \textit{$\delta$-boundary ball} $B_1(p)\subset X$, is a ball satisfying
\begin{equation}
	\dist_{GH}\left( B_1(p), B_1^{\setR^N_+}(0)  \right) \le \delta \, .
\end{equation}
We will assume that the following conditions are met:
\begin{itemize}
	\item[(H1)] The doubling $(\hat X, \hat \dist, \haus^N)$ of $X$ obtained by gluing along the boundary is an $\RCD(-\delta(N-1),N)$ space.
	
	\item[(H2)] A Laplacian comparison for the distance from the boundary holds:
	\begin{equation}\label{eq:(H2)}
		\Delta \dist_{\partial X} 
		\le - \delta(N-1) \dist_{\partial X}
		\qquad
		\text{on $X\setminus \partial X$}\, .
	\end{equation}
\end{itemize}
It is still unknown whether (H1) and (H2) hold true in the $\RCD$ class. However, they are satisfied on Alexandrov spaces, see \cite{Perelman91,AlexanderBishop03,Petrunin07}, and noncollapsed GH-limits of manifolds with convex boundary and Ricci curvature bounded from below in the interior, see \cite{BrueNaberSemola22}.	
\smallskip

We shall denote by
\begin{equation}\label{eq:Vr}
	V_r(s) : = \frac{\Leb^N(B_r((0,s))\cap \{x_N>0\})}{\omega_N r^N} \, ,
\end{equation}
where $(0,s)\in \setR^{N-1}\times \setR_+$. Moreover, we set 
\begin{equation}\label{eq:defgammaN}
\gamma(N) := \omega_{N-1}\int_0^1(1-V_1(t))\di t\, .
\end{equation}

\smallskip
Under the assumptions above, our main result is the following.

\begin{theorem}\label{thm:boundarymain}
	Let $(X,\dist,\haus^N)$ be an $\RCD(-(N-1),N)$ space with boundary satisfying (H1) and (H2). Let $p\in X$ and assume $\haus^N(B_1(p))\ge v>0$. Then
	\begin{equation}
		\mu_r(B_2(p))\le C(N,v) \, ,
		\quad
		\text{for any $r>0$} \, .
	\end{equation}
    Moreover, 
    \begin{equation}\label{eq:weaklimitmurBdary}
    	\mu_r \weakto  \gamma(N) \haus^{N-1}\res \partial X\, ,
    	\qquad
    	\text{in $B_1(p)$ as $r\downarrow 0$}\,,
    \end{equation}
where $\gamma(N)>0$ is the constant defined in \eqref{eq:defgammaN}.
   
\end{theorem}
The proof of \autoref{thm:boundarymain} is based on an $\eps$-regularity theorem for the metric measure boundary on $\delta$-boundary balls, the $\eps$-regularity \autoref{prop:boundary on regular balls} for regular balls and the boundary-interior decomposition \autoref{thm:decompositiontheorem}. Below we state the $\eps$-regularity theorem for boundary balls, and use it to complete the proof of \autoref{thm:boundarymain}. The rest of this section will be dedicated to the proof of the $\eps$-regularity theorem.
\smallskip

\begin{theorem}[$\eps$-regularity on boundary balls]\label{thm:eps-reg boundary}
For any $\eps>0$ if $\delta\le \delta(\eps,N)$ the following holds.
For any $\RCD(-\delta(N-1),N)$ space $(X,\dist,\haus^N)$ satisfying the assumptions (H1),(H2), if $B_{10}(p)\subset X$ is a $\delta$-boundary ball, then
\begin{equation}\label{eq:boundary on boundary balls}
	|\mu_r|(B_1(p)) \le C(N)\, ,
	\qquad
	\text{for any $r>0$} \, .
\end{equation}
	Moreover
	\begin{equation}\label{eq:to 0 on boundary balls}
     \limsup_{r\downarrow 0}\abs{\mu_r(B_1(p))-\gamma(N)}\le \eps\, .
	\end{equation}
\end{theorem}

Let us discuss how to complete the proof of \autoref{thm:boundarymain}, taking \autoref{thm:eps-reg boundary} for granted:
The combination of \autoref{thm:eps-reg boundary}, \autoref{prop:boundary on regular balls} and \autoref{thm:decompositiontheorem} implies that
\begin{equation}\label{z16}
	\mu_r(B_s(p)) \le C(N,v) s^{N-1} \, ,
	\qquad
	\text{for any $r>0$} \, ,
\end{equation}
where $B_s(p)$ is any ball of an $\RCD(-(N-1),N)$ space satisfying (H1) and (H2).
\smallskip

We let $\mu$ be any weak limit of a sequence $\mu_{r_i}$ with $r_i\downarrow 0$. By \autoref{thm:mms0}, $\mu$ is concentrated on $\partial X$. Moreover, by \eqref{z16}, $\mu = f \haus^{N-1}\res \partial X$, for some $f\in L^1(\partial X, \haus^{N-1})$. Indeed, $\mu$ is absolutely continuous w.r.t. $\haus^{N-1}\res \partial X$, which is locally finite by \cite{BrueNaberSemola22}.

In order to show that $f$ is constant $\haus^{N-1}$-a.e., it is sufficient to apply a standard differentiation argument via blow up, as $\partial X$ is $(N-1)$-rectifiable by \cite{BrueNaberSemola22}.
\smallskip

Let us fix $x\in \mathcal{S}^{N-1}\setminus \mathcal{S}^{N-2}$ and $\eps>0$. Given $\delta=\delta(\eps,N)>0$ as in \autoref{thm:eps-reg boundary}, we can find $r_0\le 1$ such that $B_r(x)$ is a $\delta$-boundary ball for any $r\le r_0$ by \cite[Theorem 1.4]{BrueNaberSemola22}.
Then, by \eqref{eq:to 0 on boundary balls} and scale invariance, it holds
\begin{equation*}
	\abs{\frac{\mu(B_r(x))}{r^{N-1}} - \gamma(N)}
	\le \eps \, \quad\text{for any $0<r<r_0$}\, .
\end{equation*}
Since $\haus^{N-1}(\mathcal{S}^{N-2})=0$, by the arbitrariness of $\eps>0$ and standard differentiation of measures, we deduce that 
\begin{equation*}
f(x)=\int_0^1(1-V_1(t))\di t\, , \quad \text{for $\haus^{N-1}$-a.e. $x\in \partial X$}. 
\end{equation*}

\subsection{Proof of \autoref{thm:eps-reg boundary}}

The proof is divided into several steps.\\ 
We begin by proving the uniform bound \eqref{eq:boundary on boundary balls} following the strategy of  \cite[Theorem 1.7]{LytchakKapovitchPetrunin21}. 
The idea is that, in the doubling space $\hat X$, the double of the $\delta$-boundary ball $B_2(p)$ is a $\delta$-regular ball; hence \autoref{thm:mms0} provides a sharp control on $\hat \mu_r$, the boundary measure of $\hat X$. The key observation is that $\hat \mu_r = \mu_r$ in $X\setminus B_{r}(\partial X)$ and $\mu_r(B_r(\partial X))$ is easily controlled by means of the estimate on the tubular neighborhood of $\partial X$ obtained in \cite{BrueNaberSemola22}.\\
In order to achieve \eqref{eq:to 0 on boundary balls}, we need to sharpen the estimate on $\mu_r(B_r(\partial X))$ when $r\downarrow 0$. Here we use two ingredients: 
\begin{itemize}
	\item[(1)] The control of $\delta$-boundary balls at every scale and location obtained in \cite[Theorem 8.1]{BrueNaberSemola22};
	
	\item[(2)] the Laplacian comparison (H2).
\end{itemize}
The first ingredient says that any ball $B_r(x)\subset B_2(p)$ is $\delta$-GH close to $B_r((0,\dist_{\partial X}(x)))\subset \setR_+^N$, hence the volume convergence theorem ensures that their volumes are comparable. Plugging this information in the definition of $\mu_r(B_r(\partial X))$, it is easily seen that \eqref{eq:to 0 on boundary balls} follows provided we control the $\haus^{N-1}$-measure of the level sets $\{\dist_{\partial X} = s\}$ in the limit $s\downarrow 0$. Here is where (2) comes into play. Indeed, the Laplacian bound \eqref{eq:(H2)} provides an almost monotonicity of $\haus^{N-1}(\{\dist_{\partial X} = s\})$ guaranteeing sharp controls and the existence of the limit.

\subsubsection{Proof of \eqref{eq:boundary on boundary balls}}

For any $r<10^{-10}$, we decompose
\begin{equation}
	B_1( \hat p ) = (B_1(\hat p) \cap B_{10r}(\partial X)) \cup (B_1(\hat p)\setminus B_{10r}(\partial X) ) \, ,
\end{equation}
where $\hat p\in \hat X$ is the point corresponding to $p$ in the doubling $\hat X$, and $\partial X \subset \hat X$ denotes the image of $\partial X$ through the isometric embedding $X \to \hat X$. Observe that
\begin{equation}
	\hat \mu_r(B_1(\hat p)\setminus B_{10r}(\partial X))
	= 2 \mu_r(B_1(p)\setminus B_{10r}(\partial X)) \, ,
\end{equation}
where $\hat \mu_r$ denotes the metric measure boundary in $\hat X$. Recalling from \eqref{eq:lower} and \eqref{eq:lower2} that the negative part of $\hat \mu_r$ is $O(r)$, we deduce
\begin{equation}\label{eq:split}
	\begin{split}
	|\mu_r(B_1(p)) - \mu_r(B_1(p)\cap B_{10 r}(\partial X))|
	& = \frac{1}{2} |\hat \mu_r(B_1(\hat p)\setminus B_{10r}(\partial X))|
	\\& 
	\le \frac{1}{2} |\hat \mu_r|(B_1(\hat p)) + C(N)r
	\\& \le  \eps + C(N)r \, .
	\end{split}
\end{equation}
In the last inequality above we used that $B_{10}(\hat p)$ is a $\delta$-regular ball, since $B_1(p)$ is a $\delta$-boundary ball, and \autoref{prop:boundary on regular balls}.
The tubular neighborhood estimate
\begin{equation}
	\haus^N(B_{10r}(\partial X)\cap B_1(p)) \le C(N) r \, ,
\end{equation}
proven in \cite[Theorem 1.4]{BrueNaberSemola22}, implies that 
\begin{equation}
	|\mu_r|(B_1(p)\cap B_{10 r}(\partial X))
	\le C(N) \, ,
\end{equation}
that together with \eqref{eq:split} gives \eqref{eq:boundary on boundary balls}.

\subsubsection{Proof of \eqref{eq:to 0 on boundary balls}}

For any compact set $K\subset \partial X$ and $r\ge 0$, we define
\begin{equation*}
\Gamma_{r,K}:=\left\{x\in X\, :\, \dist_{\partial X}(x)\le r\, \text{ and there exists $y\in K\cap \partial X$ with $\dist(x,\partial X)=\dist(x,y)$}\right\}\, ,
\end{equation*}
\begin{equation*}
\Sigma_{r,K}:=\left\{x\in X\, :\, \dist_{\partial X}(x)= r\, \text{ and there exists $y\in K\cap \partial X$ with $\dist(x,\partial X)=\dist(x,y)$}\right\}\, ,
\end{equation*}
and notice that $\Gamma_{r,K}=\bigcup_{0\le s\le r}\Sigma_{s,K}$.

\smallskip

We first claim that	
\begin{equation}\label{eq:error1}
\limsup_{r\downarrow 0}\abs{\mu_r(B_1(p))-\mu_r(\Gamma_{10r,B_1(p)})}\le \eps\, ,
\end{equation}
provided $B_1(p)$ is a $\delta(\eps,N)$-boundary ball.
In view of \eqref{eq:split} it is enough to check that
\begin{equation}\label{z13}
	\limsup_{r\downarrow 0}\abs{\mu_r(B_1(p)\cap B_{10r}(p))-\mu_r(\Gamma_{10r,B_1(p)})} \le  \eps\, .
\end{equation}
The elementary inclusion
\begin{equation}
	B_{1-10r}(p)\cap B_{10r}(\partial X)
	\subset 
	\Gamma_{10r, B_1(p)} 
	\subset
	B_{1+10r}(p)\cap B_{10r}(\partial X) \, ,
\end{equation}
yields
\begin{equation*}
	\begin{split}
	\limsup_{r\downarrow 0}&\abs{\mu_r(B_1(p)\cap B_{10r}(p))-\mu_r(\Gamma_{10r,B_1(p)})}
	\\& \le \limsup_{r\downarrow 0} \frac{2}{r} \haus^N((B_{1+10r}(p)\setminus B_{1-10r}(p))\cap B_{10r}(\partial X)) \, .
	\end{split}
\end{equation*}
In order to estimate the latter, we use that 
\begin{equation*}
	\nu_r := \frac{1}{r}\haus^N \res (B_2(p)\cap B_{10r}(\partial X)) \weakto \haus^{N-1}\res(\partial X\cap B_2(p))
\end{equation*}
as $r\downarrow 0$ (cf. Step 1 in the proof of \autoref{lemma:almostmonotonicity} below).
For $\Leb^1$-a.e. $\eta<10^{-10}$, it holds
\begin{equation*}
	\begin{split}
		\limsup_{r\downarrow 0}  \frac{2}{r} & \haus^N((B_{1+10r}(p)\setminus B_{1-10r}(p))\cap B_{10r}(\partial X)) 
		\\&
		\le \limsup_{r\downarrow 0} 2\nu_r(B_{1+\eta}(p)\setminus B_{1-\eta}(p))
		\\&
		=2 \haus^{N-1}((B_{1+\eta}(p)\setminus B_{1-\eta}(p))\cap \partial X) \, ,
	\end{split}
\end{equation*}
which implies
\begin{equation*}
	\limsup_{r\downarrow 0}  \frac{2}{r}  \haus^N((B_{1+10r}(p)\setminus B_{1-10r}(p))\cap B_{10r}(\partial X)) 
	\le 2\haus^{N-1}(\partial B_1(p)\cap \partial X)
	\le \eps\, .
\end{equation*}
The last inequality follows from the continuity of the boundary measure w.r.t. the GH-convergence \cite[Theorem 8.8]{BrueNaberSemola22} by assuming $\delta \le \delta(\eps, N)$.

\medskip

In virtue of \eqref{eq:error1}, in order to conclude the proof of \eqref{eq:to 0 on boundary balls} it is enough to control $\mu_r(\Gamma_{10r,B_1(p)})$. To this aim we rely on the following.

\begin{lemma}\label{lemma:almostmonotonicity}
	Let $(X,\dist, \haus^N)$ be an $\RCD(-\delta(N-1),N)$ space satisfying the conditions (H1) and (H2). Fix  $p\in X$ such that $B_1(p)$ is a $\delta$-boundary ball. Then, setting
	\begin{equation}
		f(s):= \haus^{N-1}(\Sigma_{s,\overline{B}_1(p)}) \, ,
	\end{equation}
	the following hold:
	\begin{itemize}
		\item[(a)] there exists a representative of $f$ with $f(0)=\haus^{N-1}(\overline{B}_1(p)\cap\partial X)$ and satisfying
		\begin{equation}\label{eq:almostmonot}
			f(s_1)-f(s_2) \le C(N,\delta)(s_1 - s_2) \, ,
			\qquad
			\text{for any $0\le s_2\le s_1<2$} \, ;
		\end{equation}
		\item[(b)]  
		\begin{equation}
		\haus^{N-1}(B_1(p)\cap\partial X)\le \lim_{s\downarrow 0} f(s) \le \haus^{N-1}(\overline{B}_1(p)\cap \partial X)\, .
		\end{equation}
	\end{itemize}

\end{lemma} 

\begin{proof}
Let us outline the strategy of the proof, avoiding technicalities. Given $s_2\le s_1$ the almost monotonicity of $f(s):= \haus^{N-1}(\Sigma_{s,\overline{B}_1(p)})$ between $s_2$ and $s_1$ encoded in \eqref{eq:almostmonot} will be obtained by applying the Gauss-Green theorem to the vector field $\nabla \dist_{\partial X}$ in a \emph{rectangular region} made of the gradient flow lines of $\dist_{\partial X}$ spanning the region between the \emph{horizontal} faces $\Sigma_{s_1,\overline{B}_1(p)}$, $\Sigma_{s_2,\overline{B}_1(p)}$ parallel to $\partial X$. The only boundary terms appearing will be $f(s_1)$ and $f(s_2)$, with opposite signs, as the lateral faces of the region have normal vector perpendicular to $\nabla\dist_{\partial X}$. The sought \eqref{eq:almostmonot} will follow, as the interior term in the Gauss-Green formula is almost nonpositive by the assumption (H2).\\
Several technical difficulties arise in the course of the proof. The most challenging one is the absence of a priori regularity for the rectangular region considered above, which is dealt with an approximation argument borrowed from \cite{CaffarelliCordoba93}. 

As the proof is very similar to those of \cite[Prop. 6.14, Prop. 6.15]{MondinoSemola21}, we will just list the main ingredients and briefly indicate how to combine them.

\medskip

\textbf{Ingredient 1.} The measures
\begin{equation}\label{eq:conv1}
\nu_{\eps}:=\frac{1}{\eps}\haus^N\res B_{\eps}(\partial X)
\end{equation}
weakly converge to $\haus^{N-1}\res\partial X$ as $\eps\downarrow 0$. Moreover, for $\Leb^1$-a.e. $s>0$, the sequence of measures
\begin{equation}\label{eq:conv2}
	\nu_{s,\eps}:=\frac{1}{\eps}\haus^N\res\left\{s\le \dist_{\partial X}\le s+\eps\right\}
\end{equation}
weakly converges to $\haus^{N-1}\res\left\{\dist_{\partial X}=s\right\}$ as $\eps\downarrow 0$.

The first statement can be checked by arguing as in the proof of \cite[Proposition 6.14]{MondinoSemola21}  after considering one copy of $X$ as a set of locally finite perimeter in the doubling space $\hat{X}$. If $\sigma$ denotes any weak limit of $\nu_{\eps}$, the inequality $\haus^{N-1}\res\partial X\le \sigma$ is satisfied without further conditions. The assumption (H2) enters into play in the proof of the opposite inequality. Notice that the local equi-boundedness of the family of measures $\nu_{\eps}$ follows from the tubular neighbourhood bounds in \cite[Theorem 1.4]{BrueNaberSemola22}.

The convergence of \eqref{eq:conv2} to $\haus^{N-1}\res \{\dist_{\partial X} = s\}$ is a classical statement, see for instance \cite{AmbrosioDiMarinoGigli17,BuffaComiMiranda21} for the weak convergence to the perimeter of $\{\dist_{\partial X}\le s\}$ and \cite{AmbrosioBrueSemola19,BruePasqualettoSemola19} for the identification between perimeter and $(N-1)$-dimensional Hausdorff measure.

\medskip

\textbf{Ingredient 2.} The Laplacian of $\dist_{\partial X}$ is a locally finite measure and
\begin{equation}
	\Delta\dist_{\partial X}\res\partial X=\haus^{N-1}\res\partial X \, .
\end{equation}
The same conclusion holds for $\dist_{\{\dist_{\partial X}\le s\}}$ for a.e. $s>0$.
Moreover, $\Delta\dist_{\{\dist_{\partial X}\le s\}}=\Delta\dist_{\partial X}$ on the set $\{\dist_{\partial X}>s\}$.

The first conclusion follows from \cite[Theorem 7.4]{BrueNaberSemola22}, under the condition (H2). The second conclusion can be proven by employing the coarea formula as in the proof of the convergence of \eqref{eq:conv2} and the elementary identity
\begin{equation*}
	\dist_{\{\dist_{\partial X}\le s\}}=\dist_{\partial X}-s \, ,
	\qquad
	\text{on $\{\dist_{\partial X}>s\}$} \, .
\end{equation*}

\medskip

\textbf{Proof of (a).} The bound 
\begin{equation*}
f(s)\le f(0)+C(N)s=\haus^{N-1}(\partial X\cap \overline{B}_1(p))+C(N,\delta)s\, ,\quad\text{for a.e. $0<s<2$}
\end{equation*}
can be obtained with the very same argument of the proof of \cite[Proposition 6.15]{MondinoSemola21}. Indeed, the only ingredients that are required are the Laplacian upper bound for $\dist_{\partial X}$, which is guaranteed by (H2) in the present setting, the coincidence of $\haus^{N-1}\res\partial X$ with the Minkowski content (Ingredient 1) and the identity $\Delta\dist_{\partial X}\res\partial X=\haus^{N-1}\res\partial X$ (Ingredient 2).

Analogously, we can prove that 
\begin{equation*}
f(s_1)\le f(s_2)+C(N,\delta)(s_1-s_2)\, ,
\end{equation*}
for $\Leb^1$-a.e. $0<s_2<s_1<2$. Indeed, it is sufficient to choose those $s_1,s_2$ such that the conclusions in Ingredient 1 and 2 are verified and it holds $$\Per(\{\dist_{\partial X}\le s\})=\haus^{N-1}\res\{\dist_{\partial X}=s\}\,.$$

\medskip

\textbf{Proof of (b).}
Given (a), it is easy to obtain (b). Indeed, the almost monotonicity \eqref{eq:almostmonot}, together with the condition $f(0)=\haus^{N-1}(\overline{B}_1(p)\cap \partial X)$ imply that the limit $\lim_{s\downarrow 0}f(s)$ exists and
\begin{equation*}
\lim_{s\downarrow 0}f(s)\le  f(0)= \haus^{N-1}(\overline{B}_1(p)\cap \partial X)\, .
\end{equation*}
It remains to check that 
\begin{equation}\label{eq:lowerbound}
\lim_{s\downarrow 0}f(s)\ge \haus^{N-1}(B_1(p)\cap\partial X)\, .
\end{equation}
In order to prove it, we fix any $0<t<1$ and verify that 
\begin{equation}\label{eq:lowerboundt}
\lim_{s\downarrow 0}f(s)\ge  \haus^{N-1}(B_t(p)\cap\partial X)\, .
\end{equation}
By Ingredient 1 and the coarea formula, it is easy to infer that 
\begin{equation}\label{eq:inte}
\lim_{s\downarrow 0}\frac{1}{s}\int_0^sf(r)\di r\ge  \haus^{N-1}(B_t(p)\cap\partial X)\, .
\end{equation}
Thanks to \eqref{eq:almostmonot}, \eqref{eq:inte} yields \eqref{eq:lowerboundt}. Taking the limit as $t\uparrow 1$ at the right hand side of \eqref{eq:lowerboundt} gives \eqref{eq:lowerbound}.

\end{proof}

We can now conclude the proof of \autoref{thm:eps-reg boundary}.
\\The structure theorem for $\delta$-boundary balls \cite[Theorem 8.1]{BrueNaberSemola22}, combined with the volume convergence \cite{Colding97,CheegerColding97,DePhilippisGigli18}, easily gives
\begin{equation}
	\abs{ V_r(\dist_{\partial X}(x)) -  \frac{\haus^N(B_r(x))}{\omega_N r^N} }
	\le \eps \, ,
	\qquad
	\text{for any $x \in B_{1}(p)$, and $r<10^{-10}$}\, ,
\end{equation}
provided $\delta \le \delta(\eps,N)$ is small enough, where $V_r$ was defined in \eqref{eq:Vr}.\\ 
Hence,
\begin{equation}\label{eq:error2}
	\begin{split}
		&\abs{ \mu_r(\Gamma_{10r,B_1(p)}) - \frac{1}{r}\int_{\Gamma_{10r,B_1(p)}}\left(1 - V_r(\dist_{\partial X}(x))\right)\di \haus^N(x)   }
		\\ & \qquad\qquad
		\le  \frac{\eps}{r}\haus^N(B_1(p)\cap B_{10 r}(\partial X))
		\\& \qquad\qquad
		\le C(N) \eps \, ,
	\end{split}
\end{equation}
where we use the tubular neighbourhood bound from \cite[Theorem 1.4]{BrueNaberSemola22}.\\ Employing the coarea formula and noticing that $1-V_r(t)=0$ for any $t\ge r$, we can compute
\begin{equation}\label{eq:VrcoareaSigma}
	\begin{split}
		\frac{1}{r}\int_{\Gamma_{10r,B_1(p)}}&\left(1 - V_r(\dist_{\partial X}(x))\right)\di \haus^N(x) 
		\\& =
		\frac{1}{r}\int_0^r (1-V_r(s)) \haus^{N-1}(\Sigma_{s,B_1(p)}) \di s
		\\& =
		\int_{0}^1 (1-V_r(tr)) \haus^{N-1}(\Sigma_{tr,B_1(p)}) \di t
		\\& =
		\int_{0}^1 (1-V_1(t)) \haus^{N-1}(\Sigma_{tr,B_1(p)}) \di t \, .
	\end{split}
\end{equation}
When $\delta \le \delta(\eps,N)$, \cite[Theorem 1.2]{BrueNaberSemola22} shows that
\begin{equation}\label{eq:epsregB}
	\max\left\{\abs{\haus^{N-1}(B_1(p)\cap \partial X) - \omega_{N-1}}\, , \abs{\haus^{N-1}(\overline{B}_1(p)\cap \partial X) - \omega_{N-1}}\right\}\le \eps \, .
\end{equation}
Thanks to \autoref{lemma:almostmonotonicity} and the dominated convergence theorem, from \eqref{eq:error2} and \eqref{eq:VrcoareaSigma} we deduce
\begin{equation}\label{eq:dominated}
\limsup_{r\downarrow 0}\abs{\frac{1}{r}\int_{\Gamma_{10r,B_1(p)}}\left(1 - V_r(\dist_{\partial X}(x))\right)\di \haus^N(x) - \omega_{N-1}\int_0^1(1-V_1(t))\di t }
\le C(N)\eps\, .
\end{equation}
Recalling that $\gamma(N) = \omega_{N-1}\int_0^1(1-V_1(t))\di t$, combining \eqref{eq:error1}, \eqref{eq:error2} and \eqref{eq:dominated}, we conclude
\begin{equation}
\limsup_{r\downarrow 0}\abs{  \mu_{r}(B_1(p)) - \gamma(N) }
\le  C(N)\eps\, ,
\end{equation}
hence completing the proof of \eqref{eq:to 0 on boundary balls}.


\begin{thebibliography}{GMS13}


\bibitem[AB03]{AlexanderBishop03}
\textsc{S. Alexander, R.L. Bishop:} 
\textit{$\mathcal{F}K$-convex functions on metric spaces.} 
Manuscripta Math. {\bf 110} (2003), no. 1, 115–133. 


%
\bibitem[ABS19]{AmbrosioBrueSemola19}
\textsc{L. Ambrosio, E. Bru\'e, D. Semola:}
\textit{Rigidity of the 1-Bakry-\'Emery inequality and sets of finite perimeter in RCD spaces}.
Geom. Funct. Anal., {\bf 19} (2019), n.4, 949-1001

%

\bibitem[ADMG17]{AmbrosioDiMarinoGigli17}
\textsc{L. Ambrosio, S. Di Marino, N. Gigli:} 
\textit{Perimeter as relaxed Minkowski content in metric measure spaces.} 
Nonlinear Anal. {\bf 153} (2017), 78–88.

%
\bibitem[AGMR15]{AmbrosioGigliMondinoRajala15}
\textsc{L. Ambrosio, N. Gigli, A. Mondino, T. Rajala:}
\textit{Riemannian Ricci curvature lower bounds in metric measure spaces with $\sigma$-finite measure.}
Trans. Amer. Math. Soc., {\bf 367} (2015), 4661–4701. 
%
					
\bibitem[AGS14]{AmbrosioGigliSavare14}
\textsc{L. Ambrosio, N. Gigli, G. Savar\'e:}
\textit{ Metric measure spaces with Riemannian Ricci curvature bounded from below}.
Duke Math. J., {\bf 163} (2014), 1405--1490.   
%
\bibitem[AGS15]{AmbrosioGigliSavare15}
\textsc{L. Ambrosio, N. Gigli, G. Savar\'e:}
\textit{Bakry-\'Emery curvature-dimension condition and Riemannian Ricci curvature bounds.}
                Ann. Probab., {\bf 43} (2015), 339--404. 
%				
		

%		
\bibitem[AMS14]{AmbrosioMondinoSavare14}
\textsc{L. Ambrosio, A. Mondino, G. Savar\'e}:
\textit{On the Bakry-\'Emery condition, the gradient estimates and
the local-to-global property of $\RCD^*(K,N)$ metric measure spaces}.
J.  Geom. Anal., \textbf{26} (2014), 1-33.
	
\bibitem[AMS19]{AmbrosioMondinoSavare19}
\textsc{L. Ambrosio, A. Mondino, G. Savaré:} 
\textit{Nonlinear diffusion equations and curvature conditions in metric measure spaces.} 
Mem. Amer. Math. Soc. {\bf 262} (2019), no. 1270, v+121 pp.	

			



\bibitem[BGHX21]{BrenaGigliHondaZhu21}
\textsc{C. Brena, N. Gigli, S. Honda, X. Zhu:}
\textit{Weakly non-collapsed $\RCD$ spaces are strongly non-collapsed},
preprint arXiv:2110.02420.

	
\bibitem[BNS22]{BrueNaberSemola22}
\textsc{E. Bruè, A. Naber, D. Semola:} 
\textit{Boundary regularity and stability for spaces with Ricci bounded below.} 
Invent. math. {\bf 228} (2022), no. 2, 777–891.	
	
				
					
\bibitem[BPS19]{BruePasqualettoSemola19}		
\textsc{E. Bruè, E. Pasqualetto, D. Semola:}	
\textit{Rectifiability of the reduced boundary for sets of finite perimeter over $\RCD(K,N)$ spaces,}
J. Eur. Math. Soc. (2022) doi: 10.4171/JEMS/1217.
	
	

	\bibitem[BPS20]{BruePasqualettoSemola20}		
\textsc{E. Bruè, E. Pasqualetto, D. Semola:}	
\textit{Rectifiability of $\RCD(K,N)$ spaces via $\delta$-splitting maps,}
Annales Fenn. Math., {\bf 46},  (2021),465--482. 
		
	
\bibitem[BPS21]{BruePasqualettoSemola21}
\textsc{E. Bruè, E. Pasqualetto, D. Semola:}	
\textit{Constancy of the dimension in codimension one and locality of the unit normal on $\RCD(K,N)$ spaces,}
accepted by Ann. Sc. Norm. Super. Pisa Cl. Sci., preprint arXiv:2109.12585v1.	
					
%								

\bibitem[BCM21]{BuffaComiMiranda21}
\textsc{V. Buffa, G. Comi, M. Miranda:} 
\textit{On BV functions and essentially bounded divergence measure fields in metric spaces,}
Rev. Mat. Iberoam. doi: 10.4171/RMI/1291 (2021). 


\bibitem[BGP92]{BuragoGromovPerelman92}
\textsc{Y. Burago, M. Gromov, G. Perelman:}
\textit{A. D. Aleksandrov spaces with curvatures bounded below.}
Uspekhi Mat. Nauk {\bf 47} (1992), no. 2(284), 3–51, 222; translation in 
Russian Math. Surveys {\bf 47} (1992), no. 2, 1–58 


\bibitem[CaC93]{CaffarelliCordoba93}
\textsc{L. A. Caffarelli, A. Córdoba:} 
\textit{An elementary regularity theory of minimal surfaces.} 
Differential Integral Equations {\bf 6} (1993), no. 1, 1–13.


\bibitem[CM21]{CavallettiMilman21}
\textsc{F. Cavalletti, E. Milman:} 
\textit{The globalization theorem for the Curvature-Dimension condition.} 
Invent. math. {\bf 226} (2021), no. 1, 1–137.
				
		

\bibitem[C01]{Cheeger01}
\textsc{J. Cheeger:} 
\textit{Degeneration of Riemannian metrics under Ricci curvature bounds.} 
Lezioni Fermiane. Scuola Normale Superiore, Pisa, 2001. ii+77 pp.



\bibitem[CC96]{CheegerColding96}
\textsc{J. Cheeger, T.-H. Colding:}
\textit{Lower bounds on Ricci curvature and the almost rigidity of warped products,}
Ann. of Math. (2), {\bf 144} (1996), 189--237.			
%			
 \bibitem[CC97]{CheegerColding97}
 \textsc{J. Cheeger, T.-H. Colding:}
 \textit{On the structure of spaces with Ricci curvature bounded below. I,}
 J. Differential Geom., {\bf 46} (1997), 406--480.
%		
	
 \bibitem[CC00]{CheegerColding2000b}
\textsc{J. Cheeger, T.-H. Colding:}
\textit{On the structure of spaces with Ricci curvature bounded below. III,}
J. Differential Geom., {\bf 54} (2000), 37--74.			
%					


\bibitem[CJN21]{CheegerJiangNaber21}
\textsc{J. Cheeger, W. Jiang, A. Naber:}
\textit{Rectifiability of singular sets in noncollapsed spaces with Ricci curvature bounded below,}
Ann. of Math. (2) {\bf 193} (2021), no. 2, 407–538.


\bibitem[CN15]{CheegerNaber15}
\textsc{J. Cheeger, A. Naber:}
\textit{Regularity of Einstein manifolds and the codimension 4 conjecture,} 
Ann. of Math. (2) \textbf{182} (2015), no. 3, 1093–1165.


\bibitem[C97]{Colding97}
\textsc{T.-H. Colding:}
\textit{Ricci curvature and volume convergence,}
Ann. of Math. (2) \textbf{145} (1997), no. 3, 477–501. 





%			
\bibitem[DPG16]{DePhilippisGigli16}
\textsc{G. De Philippis, N. Gigli:}
\textit{From volume cone to metric cone in the nonsmooth setting.}
Geom. Funct. Anal., {\bf 26} (2016), 1526--1587.			
%			
\bibitem[DPG18]{DePhilippisGigli18}
\textsc{G. De Philippis, N. Gigli:}
\textit{ Non-collapsed spaces with Ricci curvature bounded from below}.
J. Éc. polytech. Math., {\bf 5} (2018), 613–650.			
%



%			
\bibitem[EKS15]{ErbarKuwadaSturm15}
\textsc{M. Erbar, K. Kuwada, K.-T. Sturm:}
\textit{On the equivalence of the entropic curvature-dimension condition and Bochner’s inequality on metric measure spaces.}
Invent. Math., {\bf 201} (2015), 993--1071.		
%
%



%
\bibitem[G15]{Gigli15}
\textsc{N. Gigli:}
\textit{On the differential structure of metric measure spaces and applications.}
Mem. Amer. Math. Soc., {\bf 236} (2015), vi--91.
%
	
	


\bibitem[G18]{Gigli18}
\textsc{N. Gigli:}
\textit{Nonsmooth differential geometry: an approach tailored for spaces with Ricci curvature bounded from below.}
Mem. Amer. Math. Soc., {\bf 251} (2018), v--161.


  	    
	   	    
\bibitem[GP16b]{GigliPasqualetto16b}
\textsc{N. Gigli, E. Pasqualetto:}
\textit{Equivalence of two different notions of tangent bundle on rectifiable metric measure spaces,}
preprint arXiv:1611.09645, to appear on Comm. Anal. Geom.	   

\bibitem[Ha18]{Han18}
\textsc{B.-X. Han:}
\textit{Ricci tensor on $\RCD^*(K,N)$ spaces.}
J. Geom. Anal. {\bf 28} (2018), no. 2, 1295–1314. 



\bibitem[H19]{Honda19}
\textsc{S. Honda:}
\textit{New differential operator and non collapsed $\RCD$ spaces,}
Geom. Topol. {\bf 24} (2020), no. 4, 2127–2148.

	

\bibitem[HP22]{HondaPeng22}
\textsc{S. Honda, Y. Peng:}
\textit{A note on topological stability theorem from $\RCD$ spaces to Riemannian manifolds,}
preprint arXiv:2202.06500.



	   
\bibitem[JN16]{JiangNaber16}
\textsc{W. Jiang, A. Naber:}
\textit{$L^2$ curvature bounds on manifolds with bounded Ricci curvature,}
 Ann. of Math. (2) {\bf 193} (2021), no. 1, 107–222. 

\bibitem[KM19]{KapovitchMondino19}
\textsc{V. Kapovitch, A. Mondino:}
\textit{On the topology and the boundary of $N$-dimensional $\RCD(K,N)$ spaces,}
Geom. Topol. {\bf 25} (2021), no. 1, 445–495.



\bibitem[LN20]{LiNaber}
\textsc{N. Li, A. Naber:} 
\textit{Quantitative estimates on the singular sets of Alexandrov spaces.} 
Peking Math. J. {\bf 3} (2020), no. 2, 203–234.

\bibitem[LV09]{LottVillani}
\textsc{J. Lott, C. Villani:}
\textit{Ricci curvature for metric-measure spaces via optimal transport.}
Ann. of Math. (2), {\bf 169} (2009), 903--991.
%


\bibitem[KLP21]{LytchakKapovitchPetrunin21}
\textsc{V. Kapovitch, A. Lytchak, A. Petrunin:} 
\textit{Metric-measure boundary and geodesic flow on Alexandrov spaces.} 
J. Eur. Math. Soc. (JEMS) {\bf 23} (2021), no. 1, 29–62. 



\bibitem[MS21]{MondinoSemola21}
\textsc{A. Mondino, D. Semola:}
\textit{Weak Laplacian bounds and minimal boundaries in non-smooth spaces with Ricci curvature lower bounds,}
preprint arXiv:2107.12344v2.


\bibitem[Per91]{Perelman91}
\textsc{G. Perelman:} 
\textit{A.D. Alexandrov’s spaces with curvatures bounded from below, II,} 
preprint available at http://www.math.psu.edu/petrunin/papers/alexandrov/perelmanASWCBFB2+.pdf.

\bibitem[Per95]{Perelman95}
\textsc{G. Perelman:}
\textit{DC structure on Alexandrov space with curvature bounded from below,}
preprint available at https://anton-petrunin.github.io/papers/alexandrov/Cstructure.pdf.

\bibitem[PP96]{PerelmanPetrunin96}
\textsc{G. Perelman, A. Petrunin:}
\textit{Quasigeodesics and gradient curves on Alexandrov spaces,}
(1996) preprint available at https://anton-petrunin.github.io/papers.html.


\bibitem[P07]{Petrunin07}
\textsc{A. Petrunin:}
\textit{Semiconcave functions in Alexandrov's geometry.} 
Surveys in differential geometry. Vol. XI, 137–201, 
Surv. Differ. Geom., 11, Int. Press, Somerville, MA, 2007. 


\bibitem[P11]{Petrunin11}
\textsc{A. Petrunin:}
\textit{Alexandrov meets Lott-Villani-Sturm.} 
M\"unster J. Math., {\bf 4} (2011), 53--64.
%


\bibitem[R12]{Rajala12}
\textsc{T. Rajala:} 
\textit{Local Poincaré inequalities from stable curvature conditions on metric spaces.} 
Calc. Var. Partial Differential Equations {\bf 44} (2012), no. 3-4, 477–494.


\bibitem[Sa14]{Savare14}
\textsc{G. Savaré:}
\textit{Self-improvement of the Bakry-Émery condition and Wasserstein contraction of the heat flow in $\RCD(K,\infty)$ metric measure spaces,}
Discrete Contin. Dyn. Syst. {\bf 34} (2014), no. 4, 1641–1661. 




\bibitem[S06a]{Sturm06a}
\textsc{K.-T. Sturm:}
\textit{On the geometry of metric measure spaces I.}
Acta Math., {\bf 196} (2006), 65--131.
%		
\bibitem[S06b]{Sturm06b}
\textsc{K.-T. Sturm:}
\textit{On the geometry of metric measure spaces II.}
Acta Math., {\bf 196} (2006), 133--177.


\bibitem[VR08]{VonRenesse08}
\textsc{M.-K. Von Renesse:}
\textit{On local Poincaré via transportation.}
Math. Z., {\bf 259} (2008), 21--31.	
%
%


\bibitem[ZZ10]{ZhangZhu10}
\textsc{H.-C. Zhang, X.-P. Zhu:} 
\textit{Ricci curvature on Alexandrov spaces and rigidity theorems.} 
Comm. Anal. Geom. {\bf 18} (2010), no. 3, 503–553.
			
\end{thebibliography}
\end{document}